\documentclass{amsart}[11pt]

\title[Total positivity for cominuscule Grassmannians]
{Total positivity for cominuscule Grassmannians}



\usepackage{pstricks, pst-node, amssymb,  hyperref,multicol}
\newcommand{\cellsize}{11}
\newlength{\cellsz} 
\newsavebox{\cell}
\sbox{\cell}{\begin{picture}(\cellsize,\cellsize)
\put(0,0){\line(1,0){\cellsize}}
\put(0,0){\line(0,1){\cellsize}}
\put(\cellsize,0){\line(0,1){\cellsize}}
\put(0,\cellsize){\line(1,0){\cellsize}}
\end{picture}}
\newcommand\cellify[1]{\def\thearg{#1}\def\nothing{}%
\ifx\thearg\nothing
\vrule width0pt height\cellsz depth0pt\else
\hbox to 0pt{\usebox{\cell} \hss}\fi%
\vbox to \cellsz{
\vss
\hbox to \cellsz{\hss$#1$\hss}
\vss}}
\newcommand\tableaux[1]{\vcenter{\vbox{\let\\\cr
\baselineskip -16000pt \lineskiplimit 16000pt \lineskip 0pt
\ialign{&\cellify{##}\cr#1\crcr}}}}
\newcommand\tabl[1]{\vtop{\let\\\cr
\baselineskip -16000pt \lineskiplimit 16000pt \lineskip 0pt
\ialign{&\cellify{##}\cr#1\crcr}}}
\usepackage{tabmac}

\psset{unit=1pt, arrowsize=4pt, linewidth=.7pt}
\psset{linecolor=blue}
\newgray{grayish}{.90}
\newrgbcolor{embgreen}{0 .5 0}
\def\Le{\hbox{\rotatedown{$\Gamma$}}}
\def\vblack(#1, #2)#3{\cnode*[linecolor=black](#1, #2){3}{#3}}
\def\vwhite(#1,#2)#3{\cnode[linecolor=black,fillcolor=white,fillstyle=solid](#1,#2){3}{#3}}
\countdef\x=23
\countdef\y=24
\countdef\z=25
\countdef\t=26

\def\tbox(#1,#2)#3{
\x=#1 \y=#2
\multiply\x by 12
\multiply\y by 12
\z=\x \t=\y
\advance\z by 12
\advance\t by 12
\psline(\x,\y)(\x,\t)(\z,\t)(\z,\y)(\x,\y)
\advance\x by 6
\advance\y by 6
\rput(\x,\y){{\bf #3}}}

\font\co=lcircle10

\def\jr{\smash{\raise2pt\hbox{\co \rlap{\rlap{\char'005} \char'007}}
               \raise6pt\hbox{\rlap{\vrule height6.5pt}}
                \raise2pt\hbox{\rlap{\hskip4pt \vrule
          height0.4pt depth0pt
                width7.7pt}}}}
\def\textcross{\ \smash{\lower4pt\hbox{\rlap{\hskip4.15pt\vrule height14pt}}
                \raise2.8pt\hbox{\rlap{\hskip-3pt \vrule height.4pt depth0pt
                        width14.7pt}}}\hskip12.7pt}

\def\textelbow{\ \hskip.1pt\smash{\raise2.75pt%
               \hbox{\co \hskip 4.15pt\rlap{\rlap{\char'005} \char'007}
                \lower6.8pt\rlap{\vrule height3.5pt}
                \raise3.6pt\rlap{\vrule height3.5pt}}
                \raise2.8pt\hbox{%
                  \rlap{\hskip-7.15pt \vrule height.4pt depth0pt
width3.5pt}%
                  \rlap{\hskip4.05pt \vrule height.4pt depth0pt
width3.5pt}}}
                \hskip8.7pt}

\def\boxcross{\ \smash{\lower5.5pt\hbox{\rlap{\hskip4.1pt\vrule height11.8pt}}
                \raise0pt\hbox{\rlap{\hskip-1pt \vrule height.4pt depth0pt
                        width12pt}}}\hskip12.7pt}

\def\boxelbow{\ \hskip.1pt\smash{%
               \hbox{\co \hskip 4.15pt\rlap{\mathsurround=0pt\rlap{\mathsurround=0pt\char'005}\char'007}
                \lower6pt\rlap{\vrule height2.2pt}
                \raise3pt\rlap{\vrule height3pt}}
                \hbox{%
                  \rlap{\hskip-5.3pt \vrule height.4pt depth0pt
width1.8pt}%
                  \rlap{\hskip4.05pt \vrule height.4pt depth0pt
width2.5pt}}}
                \hskip8.7pt}

\renewcommand{\baselinestretch}{1.1}

\usepackage{amsmath, amsthm, amssymb, amsbsy}
\usepackage{amsfonts, latexsym, stmaryrd, amscd, xy}
\usepackage[mathscr]{eucal}
\usepackage{epsfig}

\newtheorem{theorem}{Theorem}[section]

\newtheorem{proposition}[theorem]{Proposition}
\newtheorem{prop}[theorem]{Proposition}
\newtheorem{lemma}[theorem]{Lemma}

\newtheorem{example}[theorem]{Example}
\newtheorem{corollary}[theorem]{Corollary}

\newtheorem{problem}[theorem]{Problem}

\newtheorem{definition}[theorem]{Definition}

\theoremstyle{remark}
\newtheorem{remark}[theorem]{Remark}
\usepackage{amsfonts}
\usepackage{xy}
\usepackage{amssymb}
\usepackage{amsmath}
\def\D{\mathcal{D}}
\def\R{\mathbb{R}}

\def\C{\mathbb{C}}

\def\B{\mathcal{B}}
\def\s{{\mathbb S}}

\def\P{\mathbb{P}}
\def\I{\mathcal{I}}

\def\wire{{\rm wire}}
\def\Le{\hbox{\rotatedown{$\Gamma$}}}

\def\mytilde{\kern-.015in\hbox{\lower.03in\hbox{\~{}}}\kern-.01in}

\def\v{\mathbf{v}}
\def\w{\mathbf{w}}

\def\Hom{\mathrm{Hom}}
\def\Sl{\mathrm{SL}}

\def\ss{\tilde{s}}

\def\Gr{{\rm Gr}}
\def\OG{{\rm OG}}
\def\Q{{\mathbb Q}}
\def\OP{{\mathbb{OP}}}
\def\LG{{\rm LG}}
\def\Fr{{\rm Fr}}

\newcommand{\commentout}[1]{}

\newcommand{\thmrefer}[1]{\renewcommand\thetheorem
  {\protect\ref{#1}}\addtocounter{theorem}{-1}}

\xyoption{all}
\CompileMatrices

\author{Thomas Lam}
\address{Department of Mathematics, Harvard University, Cambridge MA
02138 USA} \email{tfylam@math.harvard.edu, lauren@math.harvard.edu}
\thanks{T. L. was partially supported by NSF DMS--0600677.}
\author{Lauren Williams}

\begin{document}

\begin{abstract}
In this paper we explore the combinatorics of the non-negative part
$(G/P)_{\geq 0}$ of a cominuscule
Grassmannian. For each such
Grassmannian we define $\Le$-diagrams -- certain fillings of
generalized Young diagrams which are in bijection with the cells of
$(G/P)_{\geq 0}$. In the classical cases, we describe $\Le$-diagrams
explicitly in terms of pattern avoidance.  We also define a game on diagrams,
by which one can reduce an arbitrary diagram to a $\Le$-diagram. We
give enumerative results and relate our $\Le$-diagrams to other
combinatorial objects.  Surprisingly, the totally non-negative cells
in the open Schubert cell of the odd and even orthogonal
Grassmannians are (essentially) in bijection with preference
functions and atomic preference functions respectively.
\end{abstract}

\maketitle

\setcounter{tocdepth}{1}
\tableofcontents

\section{Introduction}

The classical theory of total positivity concerns matrices in which
all minors are non-negative.  While this theory was pioneered in the
1930's, interest in this subject has been renewed on account of the
work of Lusztig \cite{Lusztig3, Lusztig2}. Motivated by surprising
connections he discovered between his theory of canonical bases for
quantum groups and the theory of total positivity, Lusztig extended
this subject by introducing the totally non-negative points $G_{\geq
0}$ in an arbitrary reductive group $G$ and the totally non-negative
part $(G/P)_{\geq 0}$ of a real flag variety $G/P$. Lusztig
conjectured a cell decomposition for $(G/P)_{\geq 0}$, which was
proved by Rietsch \cite{Rietsch1}.  Cells of $(G/P)_{\geq 0}$
correspond to pairs $(x,w)$ where $x,w \in W$, $x\leq w$ in Bruhat
order, and $w$ is a minimal-length coset representative of
$W^J=W/W_J$.  Here $W_J \subset W$ is the parabolic subgroup
corresponding to $P$.


Coming from a more combinatorial perspective, Postnikov
\cite{Postnikov} explored the combinatorics of the totally
non-negative part of the type $A$ Grassmannian.  He described and
parameterized cells using certain fillings of Young diagrams by
$0$'s and $+$'s which he called {\it $\Le$-diagrams}, and which are
defined using the avoidance of the {\it $\Le$-pattern}. The
$\Le$-diagrams seem to have a great deal of intrinsic interest: they
were independently discovered by Cauchon \cite{Cauchon1} in the
context of primes in quantum algebras (see also \cite{LLR}); they are in bijection with
other combinatorial objects, such as decorated permutations
\cite{Postnikov}; and they are linked to the asymmetric exclusion
process \cite{CW}.

In this paper we use work of Stembridge \cite{Ste} and of Proctor
\cite{Pro}, to generalize $\Le$-diagrams to the case of cominuscule
Grassmannians. In this case the poset $W^J$ is a distributive
lattice and hence can be identified with the lattice of order ideals
of another poset $Q^J$. It turns out that the poset $Q^J$ can always
be embedded into a two-dimensional square lattice.  Each $w\in W^J$
corresponds to an order ideal $O_w \subset Q^J$ which can be
represented by a generalized Young diagram.  We then identify cells
of the non-negative part of a cominuscule Grassmannian with certain
fillings, called $\Le$-diagrams, of $O_w$ by $0$'s and $+$'s.
Arbitrary fillings of $O_w$ by $0$'s and $+$'s correspond to
subexpressions of a reduced expression for $w$; the $\Le$-diagrams
correspond to {\it positive distinguished subexpressions} \cite{MR}.

We give concise descriptions of $\Le$-diagrams for type $B$ and $D$
cominuscule Grassmannians in terms of pattern avoidance.
Unfortunately there does not seem to exist a concise description for
the remaining $E_7$ and $E_8$ cominuscule Grassmannians.  We also
define a game (the {\it $\Le$-game}) that one can play on diagrams
filled with $0$'s and $+$'s, by which one can go from any such
diagram to a $\Le$-diagram.

We then explore the combinatorial properties of $\Le$-diagrams. We
define type $B$ decorated permutations and show that they are in
bijection with $\Le$-diagrams.  We give some formulas and
recurrences for the numbers of $\Le$-diagrams. Finally, we show that
there are twice as many type $(B_n,n)$ $\Le$-diagrams in the open
Schubert cell as {\it preference functions} of length $n$, while
type $(D_n,n)$ $\Le$-diagrams in the open Schubert cell are in
bijection with {\it atomic preference functions} of length $n$.

{\bf Organization.} In Section \ref{sec:lusztig}, we give the
relevant background on total positivity for flag varieties, and in
Section \ref{s:comin}, we give background on cominuscule
Grassmannians. In Section \ref{s:Le}, we introduce $\Le$-diagrams,
$\Le$-moves, and the $\Le$-game.  The following  five sections are
devoted to characterizing $\Le$-diagrams for the cominuscule
Grassmannians of types $A$, $B$ and $D$. In Section \ref{DecPerms},
we review type $A$ decorated permutations and describe type B
decorated permutations, and in Section \ref{s:enumerate}, we give
enumerative results, including those on preference functions.

{\bf Acknowledgements.} We are grateful to
Frank Sottile and Alex Postnikov for interesting discussions.

\section{Total positivity for flag varieties}
\label{sec:lusztig} We recall basic facts concerning the totally
non-negative part $(G/P_J)_{\geq 0}$ of a flag variety and its cell
decomposition.


\subsection{Pinning}

Let $G$ be a semisimple linear algebraic group over $\C$ split over
$\R$, with split torus $T$.  Identify $G$ (and related spaces)
with their real points and consider them with their real topology.
Let $\Phi \subset \Hom(T,\R^*)$ the set of roots and choose
a system of positive roots $\Phi^+$.  Denote by $B^+$ the
Borel subgroup corresponding to $\Phi^+$.
Let $B^-$ be the opposite Borel subgroup $B^-$ such that
$B^+ \cap B^- = T$.
Let $U^+$ and $U^-$ be the unipotent radicals of $B^+$ and $B^-$.

Denote the set of simple roots by $\Pi = \{\alpha_i \ \vline \ i \in
I \} \subset \Phi^+$. For each $\alpha_i \in \Pi$ there is an
associated homomorphism $\phi_i : \Sl_2 \to G$, generated by
$1$-parameter subgroups $x_i(t) \in U^+$, $y_i(t) \in U^-$, and
$\alpha_i^\vee(t) \in T$.  The datum $(T, B^+, B^-, x_i, y_i; i \in
I)$ for $G$ is called a {\it pinning}.  Let $W = N_G(T)/T$ be the
Weyl group and for $w \in W$ let $\dot{w} \in N_G(T)$ denote a
representative for $w$.




\subsection{Totally non-negative parts of flag varieties}
Let $J \subset I$.  The parabolic subgroup $W_J \subset W$
corresponds to a parabolic subgroup $P_J$ in $G$ containing $B^+$.
Namely, $P_J = \sqcup_{w \in W_J} B^+ \dot{w} B^+$.  Let $\pi^J:
G/B^+ \to G/P_J$ be the natural projection.

The totally non-negative part $U_{\geq 0}^-$ of $U^-$ is defined to
be the semigroup in $U^-$ generated by the $y_i(t)$ for $t \in
\R_{\geq 0}$.  The totally non-negative part $(G/P_J)_{\geq 0}$ of
the partial flag variety $G/P_J$ is the closure of the image of
$U^-_{\geq 0}$ in $G/P_J$.

\subsection{Cell decomposition}
We have the Bruhat decompositions
\begin{equation*}
G/B^+ = \sqcup_{w \in W} B^+ \dot{w} B^+/B^+ =
    \sqcup_{w \in W} B^- \dot{w} B^+/B^+
\end{equation*}
of $G/B^+$ into $B^+$-orbits called {\it Bruhat cells}, and
$B^-$-orbits called {\it opposite Bruhat cells}.  For $v, w \in W$
define
\begin{equation*}
R_{v,w}: = B^+ \dot{w} B^+/B^+ \cap B^- \dot{v} B^+/B^+.
\end{equation*}
The intersection $R_{v,w}$ is non-empty precisely if $v \leq w$, and
in that case is irreducible of dimension $\ell(w) - \ell(v)$.  Here
$\leq$ denotes the {\it Bruhat order} (or {\it strong order}) of
$W$ \cite{BB}.  For $v, w
\in W$ with $v \leq w$, let
\begin{equation*}
R_{v,w ; >0} := R_{v,w} \cap (G/B^+)_{\geq 0}.
\end{equation*}


We write $W^J$ for the set of minimal length coset representatives
of $W/W_J$.  The Bruhat order of $W^J$ is the
order inherited by restriction from $W$.  Let $\I^J \subset W \times
W^J$ be the set of pairs $(x,w)$ with the property that $x \leq w$.
Given $(x,w) \in \I^J$, we define $P_{x,w; >0}^J := \pi^J(R_{x,w;
>0})$.  This decomposition of $(G/P_J)_{\geq 0}$ was introduced by
Lusztig \cite{Lusztig2}.  Rietsch showed that this is a cell
decomposition:


\begin{theorem}\cite{Rietsch1} \label{thm:cell}
The sets $P_{x,w;
>0}^J$ are semi-algebraic cells of dimension $\ell(w) - \ell(x)$,
giving a cell decomposition of $(G/P_J)_{\geq 0}$.
\end{theorem}
In fact the cell decomposition of Theorem \ref{thm:cell}
is a CW complex
\cite{PSW, RW}.
\section{(Co)minuscule Grassmannians}\label{s:comin}

We keep the notation of Section~\ref{sec:lusztig}. We say the
parabolic $P_J$ is {\it maximal} if $J = I \setminus \{j\}$ for some
$j \in I$.  We may then denote the parabolic by $P_j:= P_J$ and the
partial flag variety by $G/P_j:= G/P_J$, which we loosely call a
Grassmannian.  Similarly, we use the notation $\I^j$, $W^j$, $W_j$
and $W^j_{\max}$.

For a maximal parabolic subgroup $P_j$ we will call $P_j$, the flag
variety $G/P_j$, and the simple root $\alpha_j$ {\it cominuscule} if
whenever $\alpha_j$ occurs in the simple root expansion of a
positive root $\gamma$ it does so with coefficient one. Similarly,
one obtains the definition of {\it minuscule} by replacing roots
with coroots.  The (co)minuscule Grassmannian's have been classified
and are listed below, with the corresponding Dynkin diagrams (plus
choice of simple root) shown in Figure~\ref{fig:comin}.

\begin{proposition}
The maximal parabolic $P_j$, the flag variety $G/P_j$, and the
simple root $\alpha_j$ are (co)minuscule if we are in one of the
following situations:
\begin{enumerate}
\item
$W = A_n$ and $j \in [1,n]$ is arbitrary
\item
$W = B_n$ (or $C_n$) and $j = 1$ or $n$
\item
$W = D_n$ (with $n \geq 4$) and $j = 1, n-1$ or $n$
\item
$W = E_6$ and $j = 1$ or $6$
\item
$W = E_7$ and $j = 1$.
\end{enumerate}
\end{proposition}

\begin{figure}
\begin{tabular}{||c|c|c||}\hline
\text{Root system} & \text{Dynkin Diagram} &\text{Grassmannian} \\\hline \hline $A_n$ &
\setlength{\unitlength}{3mm}
\begin{picture}(11,3)
\multiput(0,1.5)(2,0){6}{$\circ$}
\multiput(0.55,1.85)(2,0){5}{\line(1,0){1.55}}
\put(6,1.5){$\bullet$} \put(0,0){$1$} \put(2,0){$2$}
\put(3.5,0){$\cdots$} \put(6,0){$j$} \put(7.5,0){$\cdots$}
\put(10,0){$n$}
\end{picture}
&
the usual Grassmannian $\Gr_{j,n+1}$
\\ %
\hline $B_n, n\geq 2$ & \setlength{\unitlength}{3mm}
\begin{picture}(11,3) \multiput(2,1.5)(2,0){5}{$\circ$}
\multiput(0.55,1.85)(2,0){4}{\line(1,0){1.55}}
\multiput(8.55,1.75)(0,.2){2}{\line(1,0){1.55}} \put(8.85,1.53){$>$}
\put(0,1.5){$\bullet$}
\put(10,1.5){$\circ$} \put(0,0){$1$} \put(2,0){$2$}
\put(4,0){$\cdots$} \put(7,0){$\cdots$} \put(10,0){$n$}
\end{picture} &

the odd dimensional quadric $\Q^{2n-1}$
\\
\hline $B_n, n\geq 2$ & \setlength{\unitlength}{3mm}
\begin{picture}(11,3) \multiput(2,1.5)(2,0){5}{$\circ$}
\multiput(0.55,1.85)(2,0){4}{\line(1,0){1.55}}
\multiput(8.55,1.75)(0,.2){2}{\line(1,0){1.55}} \put(8.85,1.53){$>$}
\put(0,1.5){$\circ$}
\put(10,1.5){$\bullet$} \put(0,0){$1$} \put(2,0){$2$}
\put(4,0){$\cdots$} \put(7,0){$\cdots$} \put(10,0){$n$}
\end{picture} &

odd orthogonal Grassmannian $\OG_{n,2n+1}$
\\
\hline $C_n, n\geq 2$ & \setlength{\unitlength}{3mm}
\begin{picture}(11,3) \multiput(2,1.5)(2,0){5}{$\circ$}
\multiput(0.55,1.85)(2,0){4}{\line(1,0){1.55}}
\multiput(8.55,1.75)(0,.2){2}{\line(1,0){1.55}} \put(8.85,1.53){$<$}
\put(0,1.5){$\bullet$}
\put(10,1.5){$\circ$} \put(0,0){$1$} \put(2,0){$2$}
\put(4,0){$\cdots$} \put(7,0){$\cdots$} \put(10,0){$n$}
\end{picture} &
the projective space $\P^{2n-1}$
\\

\hline $C_n, n\geq 2$ & \setlength{\unitlength}{3mm}
\begin{picture}(11,3) \multiput(2,1.5)(2,0){5}{$\circ$}
\multiput(0.55,1.85)(2,0){4}{\line(1,0){1.55}}
\multiput(8.55,1.75)(0,.2){2}{\line(1,0){1.55}} \put(8.85,1.53){$<$}
\put(0,1.5){$\circ$}
\put(10,1.5){$\bullet$} \put(0,0){$1$} \put(2,0){$2$}
\put(4,0){$\cdots$} \put(7,0){$\cdots$} \put(10,0){$n$}
\end{picture} &
the Lagrangian Grassmannian $\LG_{n,2n}$
\\ \hline %
$D_n, n\geq 4$ & $\begin{array}{c} \setlength{\unitlength}{2.9mm}
\setlength{\unitlength}{2.9mm} \begin{picture}(11,3.5)
\multiput(0,1.6)(2,0){5}{$\circ$}
\multiput(0.55,2)(2,0){4}{\line(1,0){1.55}}
\put(8.5,1.95){\line(2,-1){1.55}} \put(8.5,1.95){\line(2,1){1.55}}
\put(10,2.5){$\circ$} \put(10,0.7){$\circ$} \put(0,0){$1$}
\put(2,0){$2$} \put(4,0){$\cdots$} \put(7,0){$\cdots$}
\put(9.1,0){$n\!-\!1$} \put(11, 2.3){$n$} \put(10,2.45){$\circ$}
\put(0,1.6){$\bullet$}
\put(10,0.75){$\circ$}
\end{picture}
\end{array}$ &
the even dimensional quadric $\Q^{2n-2}$

\\ \hline %
$D_n, n\geq 4$ & $\begin{array}{c} \setlength{\unitlength}{2.9mm}
\setlength{\unitlength}{2.9mm} \begin{picture}(11,3.5)
\multiput(0,1.6)(2,0){5}{$\circ$}
\multiput(0.55,2)(2,0){4}{\line(1,0){1.55}}
\put(8.5,1.95){\line(2,-1){1.55}} \put(8.5,1.95){\line(2,1){1.55}}
\put(10,2.5){$\circ$} \put(10,0.7){$\circ$} \put(0,0){$1$}
\put(2,0){$2$} \put(4,0){$\cdots$} \put(7,0){$\cdots$}
\put(9.1,0){$n\!-\!1$} \put(11, 2.3){$n$} \put(10,2.45){$\bullet$}
\put(10,0.75){$\bullet$}
\end{picture}
\end{array}$ &
even orthogonal Grassmannian $\OG_{n+1,2n+1}$
\\ \hline
$E_6$ & \setlength{\unitlength}{3mm}
\begin{picture}(9,3.6)
\multiput(0,0.5)(2,0){5}{$\circ$}
\multiput(0.55,0.95)(2,0){4}{\line(1,0){1.6}} \put(0,0.5){$\bullet$} \put(8,0.5){$\bullet$}
\put(8,0.5){$\circ$} \put(4,2.6){$\circ$}
\put(4.35,1.2){\line(0,1){1.5}} \put(0,-.6){$1$} \put(2,-0.6){$3$}
\put(4,-.6){$4$} \put(6,-.6){$5$} \put(5,2.5){$2$} \put(8,-.6){$6$}
\end{picture} &
the real points of the Cayley plane $\OP^2$
\\[1mm] \hline %
$E_7$ & \setlength{\unitlength}{3mm}
\begin{picture}(11,4)
\put(0,0.9){$\circ$} \multiput(2,0.9)(2,0){4}{$\circ$}
\put(10,0.9){$\bullet$}
\multiput(0.55,1.35)(2,0){5}{\line(1,0){1.6}} \put(10,0.9){$\circ$}
\put(4,3){$\circ$} \put(4.35,1.6){\line(0,1){1.5}} \put(0,-.2){$1$}
\put(2,-0.2){$3$} \put(4,-.2){$4$} \put(6,-.2){$5$} \put(5,2.9){$2$}
\put(8,-.2){$6$} \put(10,-.2){$7$}
\end{picture} & the (real) Freudenthal variety $\Fr$\\[1mm] \hline
\end{tabular}
\caption{\label{fig:comin} The (co)minuscule parabolic quotients}
\end{figure}

For more details concerning this classification we refer the reader
to \cite{BL}.



Besides the Bruhat (strong) order, we also have the {\it weak order}
on a parabolic quotient (see \cite{BB} for details).
An element $w \in W$ is {\it fully commutative} if every pair of reduced words for $w$ are related by
a sequence of relations of the form $s_i s_j = s_j s_i$.  The following result is due to Stembridge \cite{Ste} (part of the statement is due to Proctor~\cite{Pro}).

\begin{theorem}
If $(W,j)$ is (co)minuscule then $W^j$ consists of fully commutative
elements.  Furthermore the weak order $(W^j, \prec)$ and strong
order $(W^j, <)$ of $W^j$ coincide, and this partial order is a
distributive lattice.
\end{theorem}

Since $(W^j,\prec)$ and $(W^j,<)$ coincide, we will just refer to
this partial order as $W^j$.  We indicate in Figure~\ref{fig:posets}
(mostly taken from \cite{LS}) the posets $Q^j$ such that $W^j =
J(Q^j)$, where $J(P)$ denotes the distributive lattice of order
ideals in $P$.  Note that the posets are drawn in ``French" notation so that
minimal elements are at the bottom left.
The diagrams should be interpreted as follows: each
box represents an element of the poset $Q^j$, and if $b_1$ and $b_2$
are two adjacent boxes such that $b_1$ is immediately to the left or
immediately below $b_2$, we have a cover relation $b_1 \lessdot b_2$
in $Q^j$. The partial order on $Q^j$ is the transitive closure of
$\lessdot$. (In particular the labeling of boxes shown in
Figure~\ref{fig:posets} does not affect the poset structure.)

We now state some facts about $Q^j$ which can be found in
\cite{Ste}.  Let $w_0^j \in W^j$ denote the longest element in
$W^j$.  The simple generators $s_i$ used in a reduced expression
for $w_0^j$ can be used to label  $Q^j$ in a way which
reflects the bijection between the minimal length coset
representatives $w \in W^j$ and (lower) order ideals $O_w \subset
Q^j$.  Such a labeling is shown in Figure~\ref{fig:posets};
the label $i$ stands for the simple reflection $s_i$. If $b \in O_w$
is a box labelled by $i$, we denote the simple generator labeling
$b$ by $s_b:= s_i$; the corresponding index $i \in I$ is the {\it
simple label} of $b$.

\begin{center}
\begin{figure}
\begin{tabular}{||c|c||}\hline
Parabolic quotient & $Q^j$ \\
\hline
$W = A_{n-1}$& $\tableaux{\\1&2&3&4&5 \\
2&3&4&5&6 \\
3&4&5&6&7\\&} $  ($n = 8$ and $j = 3$)
\\ \hline
$W = B_n$ and $j = 1$ & $\tableaux{\\1&2&3&4&3&2&1\\&}$ ($n = 4$)
\\ \hline
$W = B_n$ and $j = n$& $\tableaux{\\1&2&3&4\\2&3&4\\3&4\\4\\&}$  ($n
= 4$)
\\ \hline
$W = D_n$ and $j = 1$ & $\tableaux{\\ & & 5&3&2&1\\1&2&3&4\\&}$ ($n
= 5$)
\\ \hline
$W = D_n$ and $j = n$& $\tableaux{\\1&2&3&4\\2&3&5\\3&4\\5\\&}$  ($n
= 5$)
\\ \hline
$W = E_6$ and $j = 1$ & $\tableaux{\\&&&1&3&4&5&6\\
&&&3&4&2\\
&&2&4&5\\
1&3&4&5&6\\& }$
\\ \hline
$W = E_7$ and $j = 1$ &$ \tableaux{\\&&&&&&&&7\\&&&&&&&&6\\&&&&&&&&5\\&&&&&&&2&4\\&&&&7&6&5&4&3\\&&&&6&5&4&3&1\\&&&&5&4&2&\\&&&2&4&3\\7&6&5&4&3&1\\&}$ \\
\hline
\end{tabular}
\caption{Underlying posets of parabolic quotients}
\label{fig:posets}
\end{figure}
\end{center}

Given this labeling,
if $O_w$ is an order ideal in $Q^j$,
the set of linear extensions
$\{e: O_w \to [1,\ell(w)]\}$ of $O_w$ are in bijection with the
reduced words $R(w)$ of $w$: the reduced word (written down from right to
left) is obtained by reading the labels of $O_w$ in the order
specified by $e$.  We will call the linear extensions of $O_w$ {\it
reading orders}. (Alternatively, one may think of a linear extension
of $O_w$ as a standard tableau with shape $O_w$.)

\begin{remark} We use the following conventions: we do not distinguish
between the root systems $B_n$ and $C_n$ since we are only
interested in the posets $W^j$ (thus we refer to the Weyl group of
both root systems as $B_n$); for $W = D_n$ we always pick $j = 1$ or
$n$ since the case $j = n-1$ is essentially the same as the case $j
= n$; similarly for $W = E_6$ we always pick $j = 1$.
\end{remark}

\begin{remark}
In the literature, the two cases {\it minuscule} and {\it
cominuscule} are usually distinguished.  This distinction will not
be important for our applications.
\end{remark}

\section{\protect\Le-diagrams, \protect\Le-moves and the \protect\Le-game}\label{s:Le}
\subsection{Positive distinguished subexpressions}

In this subsection we give background on distinguished and positive
distinguished subexpressions; for more details, see
\cite{Deo:Decomp} and \cite{MR}. Consider a reduced expression in $W$, say $s_3
s_2 s_1 s_3 s_2 s_3$ in type $A_3$.  We define a {\it subexpression}
to be a word obtained from a reduced expression by replacing some of
the factors with $1$. For example, $s_3 s_2\, 1\, s_3 s_2\, 1$ is a
subexpression of $s_3 s_2 s_1 s_3 s_2 s_3$.
Given a reduced expression $\w:=s_{i_1} s_{i_2} \dots s_{i_n}$ for $w$,
we set $w_{(k)}:= s_{i_1} s_{i_2} \dots s_{i_k}$ if $k \geq 1$
and $w_{(0)} = 1$.  The following definition was given in \cite{MR}
and was implicit in \cite{Deo:Decomp}.

%
%

\begin{definition}[Positive distinguished subexpressions]
Let $\w:= s_{i_1},\dots, s_{i_n}$ be a reduced expression. We call a
subexpression $\v$ of $\w$ {\it positive distinguished} if
 \begin{equation}\label{e:PositiveSubexpression}
v_{(j-1)}< v_{(j-1)}s_{i_j}
 \end{equation}
for all $j=1,\dotsc,n$.
\end{definition}

Note that \eqref{e:PositiveSubexpression} is equivalent to
$v_{(j-1)}\le v_{(j)}\le v_{(j-1)}s_{i_j}$.  We will refer to a
positive distinguished subexpression as a PDS for short.

\begin{lemma}\label{l:positive}\cite{MR}
Given $v\le w$ in $W$ and a reduced expression $\w$ for $w$,
there is a unique PDS $\v_+$ for $v$ in $\w$.
\end{lemma}


\subsection{$\oplus$-diagrams and \protect\Le-diagrams}

The goal of this section is to identify the PDS's
with certain fillings of the boxes of order ideals of $Q^j$.

Let $O_w$ be an order ideal of $Q^j$, where $w \in W^j$.

\begin{definition}
An $\oplus$-diagram (``o-plus diagram") of shape $O_w$ is a filling of the boxes of $O_w$
with the symbols $0$ and $+$.
\end{definition}
Clearly there are $2^{\ell(w)}$ $\oplus$-diagrams of shape $O_w$.
The value of an $\oplus$-diagram $D$ at a box $x$ is denoted $D(x)$.
Let $e$ be a reading order for $O_w$; this gives rise to a reduced
expression $\w = \w_e$ for $w$.  The $\oplus$-diagrams $D$ of shape
$O_w$ are in bijection with subexpressions $\v(D)$ of $\w$: we will
make the seemingly unnatural specification that if a box $b \in O_w$
is filled with a $0$ then the corresponding simple generator $s_b$
is present in the subexpression, while if $b$ is filled with a $+$
then we omit the corresponding simple generator.  The subexpression
$\v(D)$ in turn defines a Weyl group element $v:= v(D) \in W$.

\begin{example}
Consider the order ideal $O_w$ which is $Q^j$ itself for type
$A_{4}$ with $j=2$.  Then $Q^j$ is the following poset
\begin{equation*}
\tableau[sbY]{1 &2 &3   \\ 2 &3 &4 }.
\end{equation*}

Let us choose the reading order (linear extension) indicated by the
labeling below:
\begin{equation*}
\tableau[sbY]{4 &5 &6  \\ 1 &2 &3 }.
\end{equation*}

Then the $\oplus$-diagrams
\begin{equation*}
\tableau[sbY]{0 &0 &0  \\ 0 &0 &0 }\ \ \ \tableau[sbY]{0 &+ &0  \\ 0
&0 &+ } \ \ \ \tableau[sbY]{0 &0 &0  \\ + &0 &+ }
\end{equation*}
correspond to the expressions $s_3 s_2 s_1 s_4 s_3 s_2$, $s_3 1 s_1
1 s_3 s_2$ and $s_3 s_2 s_1 1 s_3 1$.  The first and the last are
PDS's while the second one is not, since it is not reduced.
\end{example}

We next show that $v(D)$ does not depend on the linear extension
$e$.  The following statement can be obtained by inspection.

\begin{lemma}\label{lem:easycommute}
If $b, b' \in O_w$ are two incomparable boxes,  $s_b$ and
$s_{b'}$ commute.
\end{lemma}

Lemma \ref{lem:easycommute} implies the following statement.

\begin{proposition}\label{prop:independence}
Let $D$ be an $\oplus$-diagram.  Then
\begin{enumerate}
\item
the element $v:=v(D)$ is independent of the choice of reading
word $e$.
\item
whether $\v(D)$ is a PDS depends only on $D$ (and not $e$).
\end{enumerate}
\end{proposition}

\begin{proof}
For part (1), note that two linear extensions of the same poset
(viewed as permutations of the elements of
the poset) can
be connected via transpositions of pairs of incomparable elements.
By Lemma \ref{lem:easycommute}, $v(D)$ is therefore independent
of the choice of reading word.

Suppose $D$ is an $\oplus$-diagram of shape $O_w$, and consider the
reduced expression $\w:=\w_e = s_{i_1} \dots s_{i_n}$ corresponding
to a linear extension $e$. Suppose $\v(D)$ is a PDS of $\w$.
For part (2), it suffices to show that if we swap the $k$-th and
\mbox{$(k+1)$-st} letters of both $\w$ and $\v(D)$, where these
positions correspond to incomparable boxes in $O_w$, then the
resulting subexpression $\v'$ will be a PDS of the resulting reduced
expression $\w'$.  If we examine the four cases (based on whether
the $k$-th and $(k+1)$-st letters of $\v(D)$ are $1$ or $s_{i_k}$)
it is clear from the definition  that $\v'$ is a PDS.
\end{proof}

Proposition \ref{prop:independence} allows us to make the following
definition.

\begin{definition}
A $\Le$-diagram of shape $O_w$ is an $\oplus$-diagram $D$ of shape
$O_w$ such that $\v(D)$ is a PDS.
\end{definition}

The following statement follows immediately from Lemma
\ref{l:positive} and Theorem \ref{thm:cell}.

\begin{prop}\label{prop:cells}
The cells of $(G/P_j)_{\geq 0}$ defined in Theorem \ref{thm:cell} are in
bijection with pairs $(D,O_w)$ where $O_w$ is an order ideal in
$Q^j$ and $D$ is a $\Le$-diagram of shape $O_w$. Furthermore, the
cell labeled by $(D,O_w)$ is isomorphic to $(\R^+)^s$ where $s$ is
the number of $+$'s in $D$.
\end{prop}

Let us now state one of the main aims of this work.

\begin{problem}
Give a compact description of $\Le$-diagrams.
\end{problem}

\subsection{The \protect\Le-game}
Let $D$ be an $\oplus$-diagram of shape $O_w$ corresponding to an
element $v(D) \in W$.  By Lemma \ref{l:positive} and
Proposition~\ref{prop:independence} there is a unique $\Le$-diagram
$D_+$ with $v(D_+) = v(D)$.  We call $D_+$ the $\Le$-ification of
$D$.

\begin{problem}\label{prob:moves}
Describe how to produce $D_+$ from $D$.
\end{problem}

Our solution to Problem \ref{prob:moves} will be algorithmic,
involving a series of game-like moves.  Suppose $C \subset O_w$ is a
convex subset: that is, if $x$ and $y$ are in $C$ then any $z$ such
that $x<z<y$ must also be in $C$. We may extend the definition of
$\oplus$-diagrams to $C$.  In addition Proposition
\ref{prop:independence} still holds for $\oplus$-diagrams of shape
$C$.  If $D$ is an $\oplus$-diagram of shape $C$ we again denote by
$v(D) \in W$ the corresponding Weyl group element.  If $S: C \to
\{0,+,?\}$ is a filling of $C$ with the symbols $0$, $+$ and $?$, we
say that an $\oplus$-diagram $D$ is compatible with $S$ if for every
$x \in C$
\begin{enumerate}
\item
$D(x) = 0 \implies S(x) \in \{0, ?\}$, and
\item
$D(x) = + \implies S(x) \in \{+, ?\}$.
\end{enumerate}

If $x,y \in O_w$ are two boxes we let $(x,y) = \{z \in O_w \mid x <
z < y\}$ be the open interval between $x$ and $y$. Similarly, define
the half open intervals $(x,y]$ and $[x,y)$.

\begin{definition}
A $\Le$-move $M$ is a triple $(x,y,S)$ consisting of a pair $x <y
\in O_w$ of comparable, distinct boxes together with a filling of
the open interval $S:
(x,y) \to \{0,+,?\}$ such that
\begin{equation}\label{eq:Le}
v(D \cup x) = v(D \cup y)
\end{equation}
for every $\oplus$-diagram $D$ of shape $(x,y)$ compatible with $S$.
Here $D \cup x$ ($D \cup y$) is the $\oplus$-diagram of shape
$[x,y)$ ($(x,y])$ obtained from $D$ by placing a $0$ in $x$ ($y$).
We say that $(x,y,S)$ is a $\Le$-move from $y$ to $x$ via $S$.

Now if $D$ is an $\oplus$-diagram whose shape contains $[x,y]$, we
say that a $\Le$-move $M = (x,y,S)$ {\it can be performed} on $D$ if $D(y)
= 0$ and $D|_{(x,y)}$ is compatible with $S$.  The result of $M$ on
$D$ is then the $\oplus$-diagram $D'$ obtained from $D$ by setting
$D(y) = +$ and switching the entry of $D(x)$ (that is, $D'(x) = 0$
if $D(x) = +$ and $D'(x)= +$ if $D(x) = 0$).

\end{definition}

\begin{remark}\label{rem:Le}
Let the simple generator corresponding to the box $x$ (resp. $y$) be
the simple root $\alpha$ (resp. $\beta$).  Then (\ref{eq:Le}) is
equivalent to $v(D) s_\alpha = s_\beta v(D)$ which in turn is
equivalent to
\begin{equation}\label{eq:LeCheck}
v(D)^{-1} \cdot \beta = \alpha. \end{equation}
\end{remark}

For two $\{0,+,?\}$-fillings $S,S'$ of the same shape let us say
that $S'$ is a specialization of $S$ (and $S$ a generalization of
$S'$) if $S'$ is obtained from $S$ by changing some $?$'s to $0$'s
or $+$'s. It is then clear from the definition that if $S'$ is a
specialization of $S$ and $(x,y,S)$ is a $\Le$-move then so is
$(x,y,S')$.

The following lemma is
immediate from the definitions.

\begin{lemma}
If $D'$ is obtained from $D$ by a sequence of $\Le$-moves,
$v(D') = v(D)$.
\end{lemma}

Performing a $\Le$-move on an $\oplus$-diagram $D$ either reduces
the number of $0$'s or moves a $0$ to a box which is smaller in
the partial order (and the $+$ to
a bigger box).  Thus any sequence of $\Le$-moves must eventually
terminate.

\begin{proposition}\label{prop:greedyLe}
No $\Le$-moves can be performed on a $\Le$-diagram.  Every
$\oplus$-diagram $D$ can be $\Le$-ified by a finite sequence of
$\Le$-moves.
\end{proposition}
\begin{proof}
Let us assume that a reading order has been fixed for $O_w$ and let
$n=\ell(w)$.  It is known (\cite[Lemma 3.5]{MR}) that the unique PDS
$\v_+=t_1 t_2 \dots t_n$ for $v$ can be constructed greedily from
the right.  More precisely, we have that $v_{(n)}=v$, and once we
have determined $t_i \dots t_n$ we can determine $v_{(i-1)}$; to
construct $\v_+$ we  set
\begin{equation}
\label{eq:greedy}
t_j = \begin{cases} s_{i_j} &\mbox{if
$v_{(j)}s_{i_j} <
v_{(j)}$} \\
1 &\mbox{otherwise.}\end{cases}
\end{equation}
The application of a $\Le$-move shifts simple generators to the
right in the corresponding word.  Since $\v_+$ already corresponds
to the rightmost word, we deduce that no $\Le$-moves can be
performed on a $\Le$-diagram.

Now suppose an $\oplus$-diagram $D$ is not a $\Le$-diagram.  Let $D$
differ from its $\Le$-ification $D_+$ at a box $b$ where $b$ is
chosen to be as early as possible in the reading order.  By the
greedy property of a PDS, $D(b) = +$ and $D_+(b) = 0$.
 Denote the
set of boxes occurring after $b$ in the reading order by $A \subset
O_w$.  Then $\v(D_+)$ has the form $\v((D_+)|_A) s_b \v'$ for some $\v'$
and $\v(D)$ has the form $\v((D|_A) \v'$, which implies that
$v(D|_A)s_b = v((D_+)|_A)$ and
$v((D_+)|_A) < v(D|_A)$.  Thus by the exchange axiom, $v(D|_A)s_b$
is obtained by omitting a simple generator from $\v(D|_A)$. Let $b'$ be
the box corresponding to this simple generator;
then the $\Le$-move
$(b,b',v(D|_{(b,b')}))$ can be performed on $D$.  Repeating this, we
eventually obtain $D_+$.
\end{proof}

We say that a set $\s$ of $\Le$-moves is {\it complete} if every
$\oplus$-diagram $D$ can be $\Le$-ified using $\Le$-moves in $\s$
only.

\begin{problem}
Describe a complete set of $\Le$-moves.
\end{problem}

%
%

\section{Type $A_{n-1}$}

In this section we will give a compact description of $\Le$-diagrams
in type $A_{n-1}$ and observe that they are the same as the
$\Le$-diagrams defined by Postnikov \cite{Postnikov}.  Let $(W,j) =
(A_{n-1},j)$ so that any $O_w$ can be identified with a Young
diagram within a $j \times (n-j)$ rectangle.

\begin{theorem}\label{thm:LeA}
An $\oplus$-diagram of shape $O_w$ in type $A_{n-1}$ is a
$\Le$-diagram if and only if there is no $0$ which has a $+$ below
it and a $+$ to its left.
\end{theorem}

In Theorem \ref{thm:LeA}, ``below" means below and in the same
column, while ``to its left" means to the left and in the same row.
If an $\oplus$-diagram satisfies these condition, we say that it
possesses the $\Le$-condition.  Theorem \ref{thm:LeA} can be proved
using the wiring-diagram argument from \cite[Theorem 19.1]{Postnikov}.
This is
similar to the proof of the (much)
more difficult Theorem \ref{thm:LeD} below.  Instead, our proof
below will appeal to the fact that the cells of the type $A_{n-1}$
Grassmannians have previously been enumerated.

Let $x < y$ be two distinct, comparable boxes in $O_w$. Then $[x,y]$
is a rectangle (or as a poset, a product of chains).  Given $x < y$,
let $S_{0}$ denote the following $\{0,+,?\}$ filling of $(x,y)$:
\begin{equation}\label{eq:rectangular}
\tableau[sbY]{+ &0 &0 &0 &y \\ 0 &0 &0 &0 &0 \\
0&0&0&0&0\\x&0&0&0&+}
\end{equation}
That is, $S_{0}$ is filled with $0$'s except for the top left and
bottom right corners, where it is filled with $+$'s.

\begin{proposition}\label{prop:LeMoveA}
The triples $(x,y,S_0)$ defined above are $\Le$-moves.
\end{proposition}
We will call the $\Le$-moves $(x,y,S_0)$ the {\it rectangular}
$\Le$-moves.
\begin{proof}
For simplicity and concreteness let us suppose that the top left
hand $+$ lies on the diagonal with corresponding simple generator
$s_1$, and that the rectangle $[x,y]$ has $r\geq 2$ rows and $c \geq
2$ columns. We use the criterion for a $\Le$-move described in
Remark \ref{rem:Le}.  Note that $\alpha = \alpha_r$ and $\beta =
\alpha_c$.

Since $S_0$ has no $?$'s we need only check (\ref{eq:LeCheck}) for
$D = S_0$.  Furthermore we pick the reading order obtained by
reading the rows from left to right starting from the bottom row:
$$
\tableau[sbY]{15&16 &17 &18 & \bl \\ 10&11&12 &13 &14\\
5&6&7&8&9\\\bl &1&2&3&4}
$$
We calculate using the notation $\alpha_{ij} = \alpha_i + \cdots +
\alpha_j$,
\begin{align*}
&v(D)^{-1} \cdot \alpha_c\\
&= (s_{r+1} \cdots s_{r+c-1})(s_{r-1}s_r \cdots s_{r+c-2}) \cdots
(s_2 s_3 \cdots s_{c+1})(\hat{s}_1 s_2 s_3 \cdots s_{c-1})\alpha_c\\
& = (s_{r+1} \cdots s_{r+c-1})(s_{r-1}s_r \cdots s_{r+c-2}) \cdots
(s_2 s_3 \cdots s_{c+1}) \alpha_{2,c} \\
&=(s_{r+1} \cdots s_{r+c-1})(s_{r-1}s_r \cdots s_{r+c-2}) \cdots
\alpha_{3,c+1} \\
&= \cdots\\
&=(s_{r+1} \cdots s_{r+c-2})
\alpha_{r,r+c-2} \\
&= \alpha_r.
\end{align*}
This proves that $(x,y,S_0)$ is indeed a $\Le$-move.
\end{proof}

\begin{theorem}\label{thm:LeAcomplete}
These $\Le$-moves form a complete system of $\Le$-moves.
\end{theorem}
\begin{proof} [Proof of Theorems \ref{thm:LeA} and \ref{thm:LeAcomplete}.]
Let $D$ be an $\oplus$-diagram which  does not satisfy the
$\Le$-condition.  Let $y$ be one of the boxes closest to the bottom
left which contains a $0$ violating the $\Le$-condition.  Let $z_1$
($z_2$) be the box to the left of (below) $y$ containing a $+$ which
is closest to $y$. Let $x$ be the box which forms a rectangle with
$y$, $z_1$, and $z_2$. We claim that $D|_{(x,y)} = S_0$ as in
(\ref{eq:rectangular}).

$$
\tableau[sbY]{z_1&{0}&{0}&{0}&y \\ {}&{}&{}&{}&{0} \\ {}&{}&{}&{}&{0}\\
x & {}&{}&{}&z_2}
$$

Otherwise there is a box $t \in (x,y) - \{z_1,z_2\}$ containing a
$+$.  We pick $t$ closest to $y$.  If $t$ is not below $z_1$ then
above $t$ is a box $y'$ in the same row as $y$ such that $D(y') =
0$.  This $y'$ thus violates the $\Le$-condition and is closer to
the bottom left than $y$, a contradiction.  A similar argument holds
if $t$ is not to the left of $z_2$.  We conclude that $t$ does not
exist.

Thus the rectangular $\Le$-move $(x,y,S_0)$ can be performed on $D$.
Therefore the $\Le$-diagrams must be a subset of those
$\oplus$-diagrams which satisfy the $\Le$-condition, that is, such
that there is no $0$ which has a $+$ below it and a $+$ to its left.
But in fact it has been shown that the $\oplus$-diagrams satisfying
the $\Le$-condition are in bijection with pairs $(x,w)$ where
$x\in W$, $w\in W^J$, and $x\leq w$ \cite{Postnikov, Williams2}. Therefore the
$\Le$-diagrams must be exactly those $\oplus$-diagrams satisfying
the $\Le$-condition.  This proves Theorems \ref{thm:LeA} and
\ref{thm:LeAcomplete}.
\end{proof}

In \cite{Postnikov}, Postnikov studied the totally non-negative part
of the type A Grassmannian $(Gr_{k,n})_{\geq 0}$, and showed that it
has a cell decomposition where cells are in bijection with certain
combinatorial objects he called $\Le$-diagrams.  Postnikov's
$\Le$-diagrams are obtained from ours by reflecting in a horizontal
axis.  Since Postnikov was using the English convention for Young
diagrams whereas we are using French,
Theorem \ref{thm:LeA} shows that our definition of $\Le$-diagrams
is consistent with Postnikov's definition.

\section{Type $(B_n, n)$}
Now let $(W,j) = (B_n,n)$ so that $O_w \subset Q^j$ can be
identified with a shape (a lower order ideal)
within a staircase of size $n$.  We refer to
the $n$ boxes along the diagonal of $Q^j$ as the diagonal boxes.

\begin{theorem}\label{thm:LeB}
A type $(B_n,n)$ $\Le$-diagram is an $\oplus$-diagram $D$ of shape
$O_w$ such that
\begin{enumerate}
\item if there is a $0$ above (and in the same column as) a $+$ then
all boxes to the left and in the same row as that  $0$ must also be
$0$'s.
\item any diagonal box containing a $0$ must have {\it only} $0$'s to the left of it.
\end{enumerate}
\end{theorem}
If an $\oplus$-diagram $D$ satisfies the conditions above we will say that it
satisfies the $\Le$-conditions.

We now provide some $\Le$-moves which will turn out to be complete.
Let $x < y$ be two distinct, comparable boxes in $O_w$ such that
$[x,y]$ is a rectangle.  Denote by $S_{0}$ the filling of $(x,y)$ as
in (\ref{eq:rectangular}).  The following result is proved in the
same manner as Proposition \ref{prop:LeMoveA}.

\begin{proposition}\label{prop:LeMoveB1}
The triples $(x,y,S_0)$ defined above are $\Le$-moves.
\end{proposition}
We will call the $\Le$-moves $(x,y,S_0)$ the {\it rectangular}
$\Le$-moves.

Now let $x < y$ be two distinct diagonal boxes, so that $[x,y]$ is
itself a staircase.  Denote by $S_1$ the following filling of
$(x,y)$:
$$
\tableau[sbY]{+&0&0&0&y \\ 0&0&0&0 \\ 0&0&0 \\ 0&0 \\ x}
$$
In other words, $S_1$ is filled with $0$'s with the exception of the
top-left corner box.

\begin{proposition}\label{prop:LeMoveB2}
The triples $(x,y,S_1)$ defined above are $\Le$-moves.
\end{proposition}
We call the $\Le$-moves $(x,y,S_1)$ diagonal $\Le$-moves.
\begin{proof}
We follow the same general strategy as in the proof of Proposition
\ref{prop:LeMoveA}, again using the row reading order.  Let us
assume that the top-left corner box of $(x,y)$ is labeled by simple
generator $s_k$. We calculate, using the notation $\alpha_{ij} =
\alpha_i + \cdots + \alpha_j$,
\begin{align*}
&v(D)^{-1} \cdot \alpha_n \\
&= (s_{n-1} s_n) \cdots (s_{k+1} \cdots s_{n-1} s_{n})(s_{k+1}
\cdots
s_{n-2} s_{n-1}) \alpha_n \\
&= (s_{n-1} s_n) \cdots (s_{k+1} \cdots s_{n-1} s_{n}) \alpha_{k+1,n} \\
& = (s_{n-1} s_n) \cdots \alpha_{k+2,n} \\
& = (s_{n-1} s_n) (\alpha_{n-1} + \alpha_n) \\
& = \alpha_n.
\end{align*}
This proves that $(x,y,S_1)$ is indeed a $\Le$-move.
\end{proof}

\begin{theorem}\label{thm:LeBcomplete}
The $\Le$-moves $(x,y,S_0)$ and $(x,y,S_1)$ form a complete system
of $\Le$-moves.
\end{theorem}

Before we prove Theorems \ref{thm:LeB} and \ref{thm:LeBcomplete}, we
recall the basic facts concerning the representation of $B_n$ as
signed permutations (see \cite{BB}). Let us identify the type
$A_{2n-1}$ Weyl group with the symmetric group $S_{\{\pm 1, \dots ,
\pm n \}}$. There is a homomorphism $\iota$ from the $B_n$ Weyl
group with generators $s_1,\dots, s_n$ to $S_{\{\pm 1, \dots , \pm n
\}}$, which sends $s_n$ to $(-1,1)$ and $s_i$ to the ``signed
transposition'' $(n-i, n-i+1) (-(n-i), -(n-i+1))$.  This map is
bijective onto the set of $\pi \in S_{\{\pm 1, \dots , \pm n \}}$
such that $\pi(i)=-\pi(-i)$, called {\it signed permutations}. The
Bruhat order on $B_n$ agrees with the order on signed
permutations inherited from type $A_{2n-1}$ Bruhat order.

The embedding $\iota: B_n \to S_{\{\pm 1, \dots , \pm n \}}$ allows
us to identify a type $(B_n, n)$ $\oplus$-diagram $D$ of shape $O_w$
with the type $(A_{2n-1},n)$ $\oplus$-diagram $\iota(D)$ of shape
$O_{\iota(w)}$ obtained by reflecting $D$ over the diagonal $y=x$.
The following observation is clear from the definitions.

\begin{equation}\label{reflectPDS}
\text{If $\v(\iota(D))$ is a PDS of $\iota(w)$ then $\v(D)$ is a PDS
of $w$.}
\end{equation}

\begin{proof}[Proof of Theorems \ref{thm:LeB} and \ref{thm:LeBcomplete}.]
Let $D$ be an $\oplus$-diagram.  If $D$ violates condition (1) of
Theorem~\ref{thm:LeB} then a rectangular $\Le$-move can be performed
on it, as in the proof of Theorem~\ref{thm:LeA}. Otherwise, suppose
$D$ violates condition (2) of Theorem~\ref{thm:LeB}.

Let $y$ be the diagonal box containing the $0$ violating condition
(2) closest to the bottom left and let $z$ be the box in the same
row as $y$ containing a $+$ and closest to $y$.  Let $x$ be the
diagonal box in the same column as $z$.  We claim that $D|_{(x,y)} =
S_1$.  Using the fact that $D$ satisfies condition (1) of Theorem
\ref{thm:LeB} we deduce that $D|_{(x,y)}$ contains only $0$'s along
the diagonal.  Using the assumption that $y$ was chosen closest to
the bottom left we then deduce that $D_{(x,y)} = S_1$.
$$
\tableau[sbY]{z&0&0&0&y \\ {}&{}&{}&{} \\ {}&{}&{} \\ {}&{} \\ x}
$$
This shows that $(x,y,S_1)$ can be performed on $D$.  Thus after a
finite sequence of the moves $(x,y,S_0)$ and $(x,y,S_1)$, the
$\oplus$-diagram $D$ can be made to satisfy the $\Le$-conditions. In
particular, a $\Le$-diagram must satisfy the $\Le$-conditions.

Conversely, suppose an $\oplus$-diagram $D$ satisfies the
$\Le$-conditions of Theorem \ref{thm:LeB}.  A comparison of the
$\Le$-conditions of Theorems \ref{thm:LeA} and \ref{thm:LeB} implies
that $\iota(D)$ (obtained by reflecting $D$ in the diagonal) is a
type $(A_{2n-1}, n)$ $\Le$-diagram. Thus $\v(\iota(D))$
is a PDS, hence by \eqref{reflectPDS}, $\v(D)$ is a PDS. Therefore
$D$ is a type $(B_n,n)$ $\Le$-diagram.
\end{proof}

\section{Type $(B_n,1)$}\label{sec:typeB}
Now let $(W,j) = (B_n,1)$ so that $O_w \subset Q^j$ can be
identified with a single row.  We call the box labeled $n$ (if
contained in $O_w$) the {\it middle} box, and any two boxes with the
same simple label {\it conjugate}.  The conjugate of the middle box
is itself.

\begin{theorem}\label{thm:LeBB}
A type $(B_n,1)$ $\Le$-diagram is an $\oplus$-diagram $D$ of shape
$O_w$ such that if there is a $0$ to the right of the middle box,
then the box $b$ immediately to the left of this $0$ and the
conjugate $b'$ to $b$ cannot both contain $+$'s.
\end{theorem}
\begin{proof}
Suppose $D$ and $D'$ are two $\oplus$-diagrams of shape $O_w$ so
that $v(D)= v(D')$.  Then the words corresponding to $\v(D)$ and
$\v(D')$ are related by relations of the form $s_i s_j = s_j s_i$
and $s_i^2 = 1$; that is, no braid relation is required.  This
readily implies the description stated.
\end{proof}

Let $x < y$ be a pair of conjugate boxes in $O_w$.  Let $S_0$ denote
the following filling of $(x,y)$:
$$
\tableau[sbY]{ {x}&{+}&{?}&{?}&{?}&{?}&?&+&y}
$$

The following claim is immediate.

\begin{proposition}
The triples $(x,y,S_0)$ defined above are $\Le$-moves.
\end{proposition}

\begin{theorem}
The $\Le$-moves $(x,y,S_0)$ form a complete system of $\Le$-moves.
\end{theorem}

\section{Type $(D_n,n)$}
Now let $(W,j) = (D_n,n)$ so that $O_w \subset Q^j$ can be
identified with a shape contained inside a staircase.  We refer to
the $n$ boxes along the diagonal of $Q^j$ as the diagonal boxes. The
distance of a box $b$ from the diagonal is the number of boxes that
$b$ is on top of, so that a diagonal box has distance $0$ from the
diagonal.

In the following we will say that a box $b$ is to the left or right
(above or below) another $b'$ if and only if they are also in the
same row (column).  We will use compass directions when the same row
or column condition is not intended.

\begin{theorem}\label{thm:LeD}
A type $(D_n,n)$ $\Le$-diagram is an $\oplus$-diagram $D$ of shape
$O_w$ such that
\begin{enumerate}
\item if there is a $0$ above a $+$ then
all boxes to the left of that $0$ must also be $0$'s.
\item
if there is a $0$ with distance $d$ from the diagonal to the right
of a $+$ in box $b$ then there is no $+$ strictly southwest of $b$
and $d + 1$ rows south of the $0$.
\item
one cannot find a box $c$ containing a $0$ and three distinct boxes
$b_1,b_2,b_3$ containing $+$'s so that $c$ has distance $d$ from the
diagonal and is to the right of $b_1$, the box $b_2$ is the box
$d+1$ rows below $b_1$, and finally $b_3$ is strictly northwest of
$b_2$ and strictly south of $b_1$.
\end{enumerate}
\end{theorem}

An $\oplus$-diagram $D$ satisfying the conditions of Theorem
\ref{thm:LeD} is said to satisfy the $\Le$-conditions.

We now provide a complete set of $\Le$-moves. Let $x < y$ be two
distinct, comparable boxes in $O_w$ such that $[x,y]$ is a
rectangle.  Denote by $S_{0}$ the filling of $(x,y)$ as in
(\ref{eq:rectangular}).  The following result is proved in the same
manner as Proposition \ref{prop:LeMoveA}.

\begin{proposition}\label{prop:LeMoveD1}
The triples $(x,y,S_0)$ defined above are $\Le$-moves.
\end{proposition}
We will call the $\Le$-moves $(x,y,S_0)$ the {\it rectangular}
$\Le$-moves.

Now let $x < y$ be two distinct boxes so that $x$ is $c$ columns
west of $y$ and $r$ rows south.  Let $y$ be distance $d$ from the
diagonal.  We suppose that $r > d+ 1$ and set $k = r - (d+1)$.
Denote by $S_1$ the following $\{0,+,?\}$-filling of $(x,y)$:
$$
\tableau[sbY]{ ?&?&+&0&0&y \\?&?&0&0&0&0 \\?&?&0&0&0&0 \\+&0&?&0&0 \\ 0&0&0&0\\  0&0&0 \\
x&0}
$$
where
\begin{enumerate}
\item
the $+$ in the row of $y$ is $k$ boxes to the left of $y$,
\item
the $+$ in the column of $x$ is $k$ boxes above $x$.  Our
assumptions imply that this $+$ is southwest of the first $+$ and is
$d+1$ rows south,
\item
the box below the first $+$ and to the right of the second $+$ is a
$?$, and
\item
the remaining boxes are filled with $0$'s except for the boxes both
west of the first $+$ and north of the second $+$.
\end{enumerate}

\begin{proposition}\label{prop:LeMoveD2}
The triples $(x,y,S_1)$ defined above are $\Le$-moves.
\end{proposition}
\begin{proof}
We follow the same general strategy as in the proof of Proposition
\ref{prop:LeMoveA}, again using the row reading order.  Let us
assume that the top-left corner box of $(x,y)$ is labeled $0$ (for
readability) and that the diagonal box below $y$ is labeled by $n$
rather than $n-1$.  We lose no generality here since there is an
automorphism of the $D_n$ Weyl group swapping $s_n$ and $s_{n-1}$
and fixing all other generators.  Our assumptions give the picture:
$$
\tableau[mbY]{0&&&{c-k}&&&c \\1&&{}&{}&{}&{}&{} \\{}&&{}&{}&{}&n-2&n \\\,d+1&&{}&{?}&{}&n-1 \\ {}&&{}&{}&n\\  {}&&{}&{} \\
r&&{m}}
$$
where the label $m$ of the diagonal box to the right of $x$ depends
on the parity of $k$, and the central $?$ is labeled $c-k+d+1$.  Let
$m^*$ denote $n$ if $m = n -1$ and vice versa.  Note also that $n =
c + d + 1$.

\def\s{S}

In the following we use the notation $\alpha_{ij} = \alpha_i +
\cdots + \alpha_j$ (with $\alpha_{i+1,i} = 0$), the notation $S_a^b
= s_{a} s_{a+1} \cdots s_b$ and also $\overline{s_j}$ to indicate a
simple generator which may or may not be present.  We assume
$k \geq 2$;  otherwise the calculation is even simpler.
\begin{align*}
&v(D)^{-1} s_{\alpha_{c}} \\
&=(\s_{r+1}^{n-2}s_m)(\s_{r-1}^{n-2}s_{m^*})\cdots (\s_{d+2}^{n-2}
s_n)(\s_{d+2}^{c-k+d} \overline{s_{c-k+d+1}}
\s_{c-k+d+2}^{n-2}s_{n-1})\\
&\hspace{25pt}(\overline{\s_{d}^{c-k+d-1}} \s_{c-k+d}^{n-2}s_n)
\cdots (\overline{\s_{1}^{c-k}}\s_{c-k+1}^{c+1})
(\overline{\s_{0}^{c-k-1}} \s_{c-k+1}^{c-1}) \alpha_c \\
& = \cdots (\overline{\s_{1}^{c-k}}\s_{c-k+1}^{c+1})
\alpha_{c-k+1,c}
\\
& = \cdots (\overline{\s_{d}^{c-k+d-1}} \s_{c-k+d}^{n-2}s_n) \cdots
\alpha_{c-k+2,c+1} \\
& = \cdots (\overline{\s_{d}^{c-k+d-1}} \s_{c-k+d}^{n-2}s_n)
\alpha_{c+d+1-k,n-2} \\
& = \cdots (\s_{d+2}^{c-k+d} \overline{s_{c-k+d+1}}
\s_{c-k+d+2}^{n-2}s_{n-1}) \alpha_{c+d+1-k,n-2} + \alpha_n \\
& = \cdots (\s_{d+2}^{n-2}
s_n) \alpha_{d+2,n} + \alpha_{c+d+2-k,n-2} \\
& = \cdots (\s_{r-1}^{n-2}s_{m^*}) \cdots \alpha_{d+3,n} + \alpha_{c+d+3-k,n-2} \\
& = \cdots (\s_{r-1}^{n-2}s_{m^*}) \alpha_{r-1,n} \ \ \ \mbox{since
$c-k + r - 1
=n-1$ we have $\alpha_{c-k+r-1,n-2} = 0$} \\
& = (\s_{r+1}^{n-2}s_m) \alpha_{r,n-2} + \alpha_m \\
& = \alpha_r.
\end{align*}
This proves that the triples $(x,y,S_1)$ are indeed $\Le$-moves.
\end{proof}
Now we define a third kind of $\Le$-move $(x,y,S_2)$.  We keep the
same assumptions and notation for $x$ and $y$ as for $(x,y,S_1)$.
However, now given $x$ and $y$ there is more than one choice for
$S_2$.  Denote by $S_2$ (one of) the following $\{0,+,?\}$-fillings
of $(x,y)$:
$$
\tableaux{ ?&?&?&?&+&0&0&y \\?&?&?&?&0&0&0&0 \\ +&0&0&0&+&0&0&0 \\
0&0&0&0&?&0&0&0 \\ 0&0&0&0&+&0&0 \\ 0&0&0&0&0&0 \\ 0&0&0&0&0\\
x&0&0&0}
$$
where
\begin{enumerate}
\item
the $+$ (called $z_1$) in the row of $y$ is $k$ boxes to the left of
$y$,
\item
the lower $+$ (called $z_2$) below $z_1$ is $k$ rows north of $x$ or
alternatively $d+1$ rows south of $y$,
\item
the remaining two $+$'s are chosen on the same but {\it any} row
strictly south of $z_1$ and north of $z_2$: one of these (called $z_4$)
is in the
same column as $z_1$ and $z_2$ while the other (called $z_3$) is in
the same column as $x$, and
\item
the remaining boxes are filled with $0$'s except for:
the boxes
which are strictly  west of $z_1$ and strictly north of $z_3$;
and the boxes between (and in the same column as) $z_2$ and $z_4$.
\end{enumerate}

The following result is proved in the same manner as Proposition
\ref{prop:LeMoveD2}.  In fact the half of the calculation below
$z_2$ is identical.
\begin{proposition}\label{prop:LeMoveD3}
The triples $(x,y,S_2)$ defined above are $\Le$-moves.
\end{proposition}

\begin{theorem}\label{thm:LeDcomplete}
The $\Le$-moves $(x,y,S_0)$, $(x,y,S_1)$ and $(x,y,S_2)$ form a
complete system of $\Le$-moves.
\end{theorem}
\begin{proof}[Proof of Theorems \ref{thm:LeD} and
\ref{thm:LeDcomplete}] We first show that the $\oplus$-diagrams
satisfying the $\Le$-conditions correspond to PDS's.  It is well
known \cite{BB} that $D_n$ Weyl group elements can be identified as
signed permutations on the $2n$ letters $\{\pm 1, \pm 2, \ldots, \pm
n\}$ which are {\it even}: that is, have an even number of signs
in positions $1$ through $n$.
This is achieved by the map $\delta$ which sends $s_n \mapsto
(1,-2)(2,-1)$ and $s_i \mapsto (n-i, n-i+1)(i-n,i-n-1)$ for $1 \leq
i \leq n-1$. Note that $\delta$ does {\it not} preserve Bruhat
order.

Using $\delta$, we obtain a type $(A_{2n-1},n)$
$\oplus$-diagram from a type $(D_n, n)$ $\oplus$-diagram $D$.  The
type $(A_{2n-1},n)$ $\oplus$-diagram can be converted to a wiring
diagram $\wire(D)$ in a $n \times n$ square ($+$'s become elbows and
$0$'s become crosses).  For example:

$$
\tableau[sbY]{+&0&+&+\\+&+&0\\0&0\\+\\ \bl \\ \bl } \ \ \
\raise50pt\hbox{$\longrightarrow$} \ \ \
\tableau[sbY]{+&0&+&+&+\\+&+&0&0&+\\0&0&0&0&+\\+&+&0&+&0\\ +&+&0&+&+ \\
\bl } \ \ \ \raise50pt\hbox{$\longrightarrow$} \ \ \
\tableau[sbY]{\bl&\bl 5&\bl 4&\bl 3&\bl 2& \bl 1\\
\bl 5&\boxelbow&\boxcross&\boxelbow&\boxelbow&\boxelbow&\bl\,-1\\
\bl 4&\boxelbow&\boxelbow&\boxcross&\boxcross&\boxelbow&\bl\,-2\\
\bl 3&\boxcross&\boxcross&\boxcross&\boxcross&\boxelbow&\bl\,-3\\
\bl 2&\boxelbow&\boxelbow&\boxcross&\boxcross&\boxelbow&\bl\,-4\\
\bl 1&\boxelbow&\boxelbow&\boxcross&\boxelbow&\boxelbow&\bl\,-5\\
\bl&\bl-1&\bl-2&\bl-3&\bl-4& \bl-5}
$$
Note that most boxes are replaced by an elbow or a cross in the same
position and the diagonal-symmetric position.  However, boxes
corresponding to the simple generator $s_n$ are replaced by a $2
\times 2$ square of boxes all containing either elbows or crosses.

The condition for a wiring diagram to be the wiring diagram of a PDS
is the following: two wires $p, q$ which cross in a square
corresponding to $b \in D$ are not allowed to touch or cross again
(as we read from northwest to southwest), except when that
touching/crossing happens in one of the two by two squares
corresponding to $s_n$.  If $p, q$ both enter a two by two square
corresponding to a diagonal square $b' \in D$, then the requirement
is instead that the effect on $\wire(D)$ of changing $b$ from a $0$
to a $+$ is not the same as the effect of changing $b'$ between a
$+$ and a $0$.

We allow touching/crossing again in that two by two square  as long
as not all four boxes are touching/crossing.

Now suppose $D$ is an $\oplus$-diagram satisfying all three
conditions of the theorem.  If $D$ does not correspond to a PDS then
by Proposition \ref{prop:greedyLe}, a $\Le$-move can be performed.
Let us, as in Proposition \ref{prop:greedyLe}, pick the
southwestern-most such $\Le$-move.  Thus we have two boxes $x$ and
$y$, where $y$ is filled with a $0$ and $x$ is southwest of $y$.

For the $\Le$-move to be valid -- i.e. for the signed permutation to
be unchanged by the $\Le$-move -- the wires which cross in box $y$ of
$\wire(D)$ must cross or touch again in box $x$.  Here if $x$ or
$y$ corresponds to a generator $s_n$ then one must consider the entire
$2 \times 2$ square of wires.  Suppose first that $y$ corresponds to
a simple generator $s_i$ for $i \neq n$, and let wires $a, b$ cross
in $y$ (we use $y$ to refer to the box in $D$ and also $\wire(D)$).
Say $y$ is in column $c$, using always the labeling of the wiring
diagram.

For $a$ and $b$ to cross again, there must be a $+$ to the left of
$y$, so by (1) there is no $+$ below.  Suppose the closest $+$ to
the left of $y$ is in column $c'$.  Let us suppose first that the
wire $a$ travels down and passes straight through the diagonal,
while the wire $b$ travels leftwards before turning at the first
$+$.  In this case, by (1) and (2), wire $a$ must make a turn in row
$c'$, resulting in a $+$ in position $(c,c')$.  However, using
conditions (2) and (3) we see that it is not possible for wire $b$
to travel below row $c$, and so can never meet $a$ again.  Now
suppose that wire $a$ does not cross the diagonal.  This is only
possible if the diagonal square $b$ of $D$ below $y$ corresponds to
simple generator $s_{n-1}$ and the diagonal square $z$ immediately
southwest of $b$ is a $+$.  Using the conditions (1),(2) and (3) we
obtain a picture similar to
$$
\tableau[sbY]{?&?&+&y\\0&0&?&0\\0&0&?&\tf 0\\0&0& \tf +}
$$
where diagonal boxes are in bold. The wires $a$ and $b$ can only
touch at the box $z$.
But setting $x = z$ is not a valid $\Le$-move, since the effect
of changing $z$ is to swap $a$ and $-b$, not to swap $a$ and $b$.

Finally, suppose $y$ corresponds to the simple generator $s_n$.
Then again there must be a $+$ to the left of $y$, and automatically
we deduce that one wire (say $a$) travels down through the diagonal
while wire $b$ travels to the left and turns at the closest $+$. The
argument for this case is the same as before: the wires $a$ and $b$
never touch again.  Thus if $D$ satisfies the $\Le$-conditions it
must be a $\Le$-diagram.

Let $D$ be an $\oplus$-diagram.  We shall show that if $D$ does not
satisfy the $\Le$-conditions then one of the $\Le$-moves
$(x,y,S_0)$, $(x,y,S_1)$ and $(x,y,S_2)$ can be applied to it, which
will complete the proof.  If $D$ violates the $\Le$-condition (1) of
Theorem~\ref{thm:LeD} then a rectangular $\Le$-move can be performed
on it, as in the proof of Theorem~\ref{thm:LeA}. Otherwise, suppose
$D$ violates either condition (2) or (3) of Theorem~\ref{thm:LeD}.

Let $y$ be the box containing the $0$ violating condition (2) or (3)
closest to the bottom left.  Suppose $y$ is distance $d$ from the
diagonal.  Let $z_1$ be the box to the left of $y$ containing a $+$
which is closest to $y$.

Suppose first that there is a box $z_2$ such that $(y,z_1,z_2)$
violates condition (2).  Pick $z_2$ rightmost with this property.
Let $z_1$ be $k$ boxes to the left of $y$ and let $x$ be the box $k$
boxes below $z_2$.  We claim that $D|_{(x,y)}$ is compatible with
$S_1$ and shall explain the claim pictorially. Using condition (1)
and the rightmost property of $z_2$ we may deduce at least the
following information:

$$
\tableau[sbY]{ ?&?&z_1&0&0&y \\?&?&?&0&0&0 \\?&?&?&0&0&0 \\z_2&0&?&0&0 \\ ?&0&?&0\\  ?&0&? \\
x&0}
$$
To deduce the location of the remaining $0$'s we need to use the
assumption that $y$ is the bottom leftmost box containing a 0
violating conditions (2) or (3).  The $0$'s to the left of $y$ allow
us to deduce that the $?$'s in the rows between $x$ and $z_2$ are
$0$'s.  The $0$'s below $y$ allow us to deduce that the $?$'s north
of $z_2$ and below $z_1$ are also $0$'s.  This shows that
$D|_{(x,y)}$ is compatible with $S_1$ and so the $\Le$-move
$(x,y,S_1)$ can be performed on $D$.

If $y$ does not participate in a pattern of type
(2) but {\it does} participate in a pattern of type 3, then
there is no $+$ southwest of $z_1$ and $(d+1)$ rows south.  Using
condition (1), there must be a box $z_2$ containing
$+$ which is $(d+1)$ rows below $z_1$, and there is a $z_3$ so that
$(y,z_1,z_2,z_3)$ violates condition (3).  We assume $z_3$ is chosen
as south and as east as possible (there may be more than one
choice).  Let $x$ be the box $k$ rows south of $z_2$ and in the same
column as $z_3$, where $z_1$ is $k$ boxes to the left of $y$.  We
claim that $D|_{(x,y)}$ is compatible with $S_2$ and shall explain
the claim pictorially.  The following information can be deduced
using the eastmost-ness of $z_1$ and $z_3$ and condition (1).

$$
\tableau[sbY]{ ?&?&?&?&z_1&0&0&y \\?&?&?&?&?&0&0&0 \\ z_3&0&0&0&z_4&0&0&0 \\
?&0&0&0&?&0&0&0 \\ ?&0&0&0&z_2&0&0 \\ ?&0&0&0&?&0 \\ ?&0&0&0&?\\
x&0&0&0}
$$
The $0$'s to the right of $z_2$ allow us to deduce that the $?$'s
above $x$ and south of $z_2$ are $0$'s.  The southmost-ness of $z_3$
and the assumption that $y$ and $z_1$ are not involved in a
violation of condition (2) gives us the remaining $0$'s between
$z_3$ and $x$.  The $0$'s between $y$ and $z_1$ and the
assumption on $y$ being as southwest as possible allows us to
deduce that the $?$'s below $z_2$ are $0$'s.  Finally,
the southwest assumption on $y$ allows us to deduce that the $?$'s between
$z_1$ and $z_4$ are $0$'s.  This shows that
$D|_{(x,y)}$ is compatible with $S_2$ and so the $\Le$-move
$(x,y,S_2)$ can be performed on $D$.

\end{proof}
\begin{remark}
Theorems \ref{thm:LeA} and \ref{thm:LeB} can also be proved using
wiring diagrams.
\end{remark}

\section{Type $(D_n,1)$}
Now let $(W,j) = (D_n,1)$ so that $O_w \subset Q^j$ can be
identified with a shape contained inside the doubled tail diamond of
size $n$.  The analysis for this case is nearly identical to type
$(B_n,1)$.

We call the boxes labeled $n-1$ and $n$ (if contained in $O_w$) the
{\it middle} boxes, and any two boxes with the same simple label
{\it conjugate}. If $b$ is a middle box then the conjugate of $b$ is
the other middle box.  The proof of the following statement is the
same as for Theorem~\ref{thm:LeBB}.

\begin{theorem}\label{thm:LeDD}
A type $(D_n,1)$ $\Le$-diagram is an $\oplus$-diagram $D$ of shape
$O_w$ such that if there is a $0$ in a box $c$ greater than the
middle boxes, then for any box $b \lessdot c$ covered by $c$ the
conjugate $b'$ and the box $b$ itself cannot both contain $+$'s.
\end{theorem}

Let $x < y$ be a pair of conjugate and comparable boxes in $O_w$.
Let $S_0$ denote the following filling of $(x,y)$:
$$
\tableau[sbY]{\bl&\bl&\bl&?&?&?&+&y\\ x&+&?&?&?}
$$
if $(x,y)$ is not adjacent to the middle boxes and
$$
\tableau[sbY]{+&y\\ x&+}
$$
otherwise.

The following claim is immediate.

\begin{proposition}
The triples $(x,y,S_0)$ defined above are $\Le$-moves.
\end{proposition}

\begin{theorem}
The $\Le$-moves $(x,y,S_0)$ form a complete system of $\Le$-moves.
\end{theorem}

\section{Decorated permutations for types $A$ and $B$}\label{DecPerms}
In \cite{Postnikov}, Postnikov defined decorated permutations,
proved that they are in bijection with type $A$ $\Le$-diagrams, and
described the partial order on totally positive Grassmannian cells
in terms of decorated permutations. We will review decorated
permutations in the type $A$ case and then describe type B decorated
permutations.

\subsection{Type A decorated permutations}

In this section we will fix $W=S_n$, the symmetric group, with
standard generators $\{s_i\}$,  $j\in \{1,\dots n\}$, and $W_J =
\{s_1, \dots , \hat{s_j},\dots ,s_n\}$.  Recall from Section
\ref{sec:lusztig} that $R^j$ denotes the poset of cells of the
corresponding Grassmannian.

A {\it decorated permutation} $\tilde{\pi} = (\pi, d)$ is a
permutation $\pi$ in the symmetric group $S_n$ together with a
coloring (decoration) $d$ of its fixed points $\pi (i)=i$ by two
colors. Usually we refer to these two colors as ``clockwise'' and
``counterclockwise'', for reasons which the partial order will make
clear.  When writing a decorated permutation in one-line notation,
we will put a bar over the clockwise fixed points.  A nonexcedance
of a (decorated) permutation $\tilde{\pi}$ is an index $i \in [n]$
such that either $\pi(i) < i$ or $\pi(i) = i$ and $i$ is labeled
clockwise.

Let $\Le(j,n)$ denote the set of type $(A_{n-1},j)$ $\Le$-diagrams
and let $\D_{j,n}$ denote the set of decorated permutations on $n$
letters with $j$ nonexcedances. Let us refer to the maximal order
ideal in $Q^j$ as the {\it maximal shape}.  So for example in type
$A_n$ the maximal shape for $Q^j$ is a $j \times (n-j)$ rectangle.

Much of the content of the following result can be found in
\cite{Postnikov, Williams2}.

\begin{theorem}\label{thm:commuteA}
There exist maps $\Phi_1, \Phi_2$, and $\Phi_3$, such that the
following diagram commutes.
\end{theorem}

\vspace{.3cm}
$$
\begin{psmatrix}
\I^j &  & \protect\Le(j,n)\\
& \D_{j,n}
\psset{nodesep=3pt, arrows=->}
\ncline{1,1}{2,2}<{\Phi_1}
\ncline{1,3}{2,2}>{\Phi_3}
\ncline{1,3}{1,1}\taput{\Phi_2}
\end{psmatrix}
$$
\vspace{.3cm}

\subsubsection{From pairs of permutations to decorated permutations}
The bijection $\Phi_1$ was stated (with slightly different conventions)
in the appendix of
\cite{Williams2}.
Let $(v,w) \in \I^j$.  Then $\Phi_1((v,w)) =
(\pi, d)$ where $\pi = vw^{-1}$.  To define the decoration $d$, we
make any fixed point that occurs in one of the positions $w(1),
w(2), \dots, w(k)$ a {\it clockwise loop} -- a nonexcedance -- and
we make any fixed point that occurs in one of the positions $w(k+1),
\dots , w(n)$ a {\it counterclockwise loop} -- a weak excedance. The
fact that $\Phi_1$ is a bijection will be established in Section
\ref{sec:commute}.

\subsubsection{From pairs of permutations to $\Le$-diagrams (and
back)}

To define $\Phi_2$, we may simply take a $\Le$-diagram $D$ of shape
$O_w$ to the pair $(v(D),w)$.  It follows from Proposition
\ref{prop:cells} that this is a bijection.

The map $\Phi_2$ can also be described as
follows (see also \cite{Postnikov}). View an $\oplus$-diagram $D$
within the maximal shape and label the unit steps of the northeast
border of the maximal shape with the numbers from $1$ to $n$ (from
northwest to southeast); we label the southwest border with the
numbers from $1$ to $n$ (from northwest to southeast). The map
$\Phi_2$ is defined by interpreting an $\oplus$-diagram $D$ as a
wiring diagram; replace each $0$ with a $\textcross$ and each $+$
with a $\textelbow$. By starting from the southwest labels of the
border and traveling northeast we can read off a permutation $v$.
If we replace all boxes with a
$\textcross$ and perform the above procedure we can read off a
permutation which we will call $w$. Then $\Phi_2(D) = (v,w)$.

\subsubsection{From $\Le$-diagrams to decorated permutations}

Now we describe $\Phi_3$. Given a $\Le$-diagram $D$ of shape $O_w$,
we label the northeast border of $O_w$ with the numbers $1$ to $n$
from northwest to southeast, just as in the definition of $\Phi_2$.
We ignore the $0$'s in the $\Le$-diagram and replace each $+$ with a
vertex as well as a ``hook" which extends north and east (ending at
the boundary of $O_w$). We let $G(D)$ denote the graph which is the
union of the hooks and the vertices. Now define a permutation $\pi$
as follows.  If $i$ is the label of a horizontal unit step, then we
start at $i$, and travel along $G(D)$, first traveling as far south
as possible, and then zigzagging east and north, turning wherever
possible (at each new vertex).  Then $\pi(i)$ is defined to be the
endpoint of this path. Clearly $\pi(i) \geq i$.  Similarly, if $i$
is the label of a vertical unit step, then we start at $i$, and
travel along $G(D)$, first traveling as far west as possible, and
then zigzagging north and east, turning wherever possible.  As
before,  $\pi(i)$ is defined to be the endpoint of this path, and
clearly $\pi(i) \leq i$. If $i$ is the step which cannot travel
anywhere then $i$ becomes a counterclockwise fixed point (weak
excedance) if the step is horizontal and$i$ becomes a clockwise
fixed point (nonexcedance) if the step is vertical.  This map was
proved to be a bijection in \cite{SW} and is a simplification of a
map found by Postnikov.

\begin{remark}
If we consider a clockwise fixed point to be a nonexcedance and a
counterclockwise fixed point to be a weak excedance, then it is
clear that $\Phi_3$ maps $\Le$-diagrams contained in a $j \times
n-j$ rectangle to permutations in $S_n$ with precisely $j$
nonexcedences. Clearly even more is true; the shape of the Young
diagram $O_w$ determines the positions of the nonexcedences and weak
excedances.
\end{remark}


\begin{example}
Consider the following $\Le$-diagram $D$ (viewed inside of the $j$ by $n-j$ rectangle
associated to the corresponding Grassmannian $Gr_{j,n}$).
\begin{equation*}
\tableaux{\  &\ &\ &\  \\ 0 & + &\ &\ \\ 0 & 0 & 0 &\ \\ + & +&+&0}
\end{equation*}
Then $\Phi_2(D) = ((1,3,6,2,4,5,8,7), (1,4,6,8,2,3,5,7)).$
To compute $\Phi_3(D)$, we construct the following graph $G(D)$.
\begin{figure}[h]
\centering
\includegraphics[height=.7in]{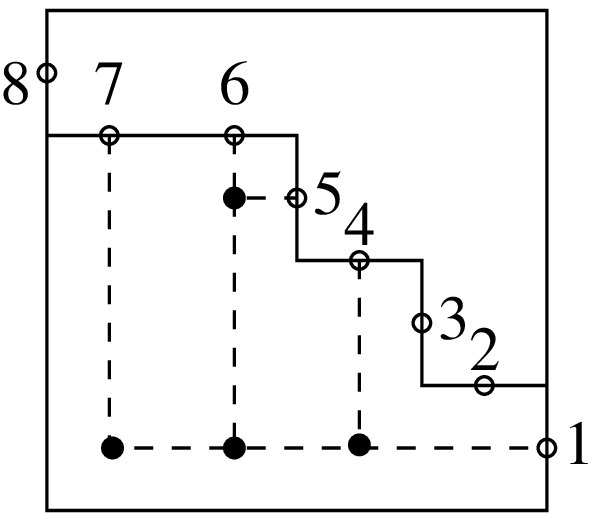}
\label{Bijection}
\end{figure}
Following the procedure explained above, the resulting permutation is
$(\overline{1},4,5,3,8,6,7,2)$.

\end{example}

\subsubsection{Proof of Theorem
\ref{thm:commuteA}}\label{sec:commute} If we compare the definition
of $\Phi_3$ to $\Phi_1$ and $\Phi_2$, it is clear that what $\Phi_3$
does is to interpret the $\Le$-diagram as two wiring diagrams (one
for $v$ and one for $w$) and then to compute $w^{-1}$ followed by
$v$.  Thus $\Phi_3 = \Phi_1 \circ \Phi_2$, proving the
commutativity.  Since $\Phi_2$ and $\Phi_3$ are bijections, so is
$\Phi_1$.

\subsection{Type B decorated permutations and type B permutation tableaux}
We now describe the type $(B_n,n)$ analogue of Theorem
\ref{thm:commuteA}. Let $\ss_0, \ss_1, \dots , \ss_{n-1}$ denote the
Coxeter generators of the type $B_n$ Coxeter group, where $\ss_0$
labels the special node of the Dynkin diagram.  {\it Note that in
this section only we are departing from the earlier notation set up
in Section \ref{s:comin} by using $0$ rather than $n$ for the
special generator.}  To avoid confusion, we refer to all objects as
``type $B_n$'' objects without specifying the cominuscule node.

Recall from Section \ref{sec:typeB} that $B_n$ Weyl group elements
can be thought of as signed permutations via the embedding $\iota$
into $S_{\{\pm 1, \dots , \pm n\}}$.  We will use the notation
$(a_1, \dots , a_n)$ to denote the signed permutation $\pi$ such
that $\pi(i) = a_i$.
%

We define a {\it $B_n$ decorated permutation} to be a signed
permutation in which fixed points are either labeled clockwise or
counterclockwise, and such that $\pi(i)$ is a clockwise fixed point
if and only if $\pi(-i)$ is a counterclockwise fixed point.  When
writing a type $B$ decorated permutation in list notation, we will
indicate a clockwise fixed point by putting a bar over the
corresponding letter.

In this section our parabolic subgroup $W_J$ will be $\{\ss_1, \dots
, \ss_{n-1}\}$.  The pairs $(v,w)$ of Theorem \ref{thm:cell} will be
denoted $\I^B$.  Let $\Le^B(n)$ denote the set of type $B$
$\Le$-diagrams contained in the maximal shape (this time an inverted
staircase), and let $\D^B_n$ denote the set of type $B$ decorated
permutations on the letters $\{\pm 1, \dots , \pm n\}$.

\begin{theorem}\label{Bcommutative}
There exist bijections $\Phi^B_1, \Phi^B_2$, and $\Phi^B_3$, such
that the following diagram commutes. \vspace{.3cm}
$$
\begin{psmatrix}
\I^B &  & \protect\Le^B(n)\\
& \D^B_n
\psset{nodesep=3pt, arrows=->}
\ncline{1,1}{2,2}<{\Phi^B_1}
\ncline{1,3}{2,2}>{\Phi^B_3}
\ncline{1,3}{1,1}\taput{\Phi^B_2}
\end{psmatrix}
$$
\vspace{.3cm}

\end{theorem}

\begin{proof}
Let $\theta_{\Le}: \Le(n,2n) \to \Le(n,2n)$ denote the involution of
type $(A_{2n-1},n)$ $\Le$-diagrams obtained by reflection in the
diagonal.  Reflecting a type $B$ $\Le$-diagram in the diagonal
defines an embedding $\iota_{\Le}: \Le^B(n) \to \Le(n,2n)$ such that
the image of $\iota_{\Le}$ is the fixed points of $\theta_{\Le}$.

Let $\I^n$ denote the parametrising set of Theorem \ref{thm:cell}
for $(A_{2n-1},n)$ where we take the symmetric group to be $S_{\{\pm
1, \dots , \pm n\}}$.  Define $\iota_\I: \I^B \to \I^n$ by
$\iota_\I(x,w) = (\iota(x),\iota(w))$ where on the right hand side
we use $\iota: B_n \to S_{\{\pm 1, \dots , \pm n\}}$.  This map
makes sense since $\iota$ is a Bruhat embedding \cite{BB}.  The
image of $\iota_\I$ is the set of fixed points of the map
$\theta_\I$, obtained by applying $\theta: S_{\{\pm 1, \dots , \pm
n\}} \to S_{\{\pm 1, \dots , \pm n\}}$ given by $\theta(\pi)(i) =
-\pi(-i)$, to each of a pair of permutations.

Similarly, define $\iota_\D : \D^B_n \to \D(n,2n)$.  Again the image
is the set of fixed points of $\theta_\D: \D(n,2n) \to \D(n,2n)$ the
map induced by $\theta$.  Here if $\pi(i) = i$ is a fixed point
labeled clockwise (resp. counterclockwise) then $\theta_\D(\pi)(-i)
= -i$ is a fixed point labeled counterclockwise (resp. clockwise).

By the diagonal reflection invariance of the definitions of $\Phi_1,
\Phi_2, \Phi_3$ in Theorem \ref{thm:commuteA} we see that the three
bijections commute with the respective involutions $\theta_{\Le},
\theta_\D$, and $\theta_\I$.  We may thus restrict Theorem
\ref{thm:commuteA} to the fixed points of $\theta_{\Le}, \theta_\D$,
and $\theta_\I$, giving the stated result.
\end{proof}


%

\begin{example}
Consider the following $\Le$-diagram $D$.
\begin{equation*}
\tableaux{+  &\ &\  \\ 0 & 0  \\ + }
\end{equation*}
This diagram corresponds to the element $(v,w)=(\ss_0 \ss_1, \ss_2
\ss_0 \ss_1 \ss_0)$ in $\I^0$, which in list notation is equal to
$((2,-1,3), (-3,-1,2))$.  To compute the corresponding decorated
permutation we construct from $D$ the following graph $G(D)$.
\begin{figure}[h]
\centering
\includegraphics[height=.7in]{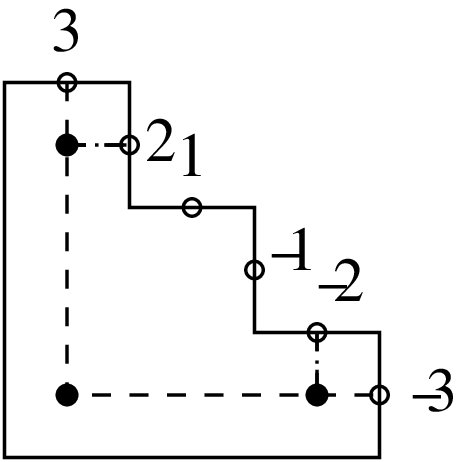}
\label{BBijection}
\end{figure}
Then the resulting decorated permutation is $(\overline{1}, 3, -2)$.
\end{example}

\subsubsection{Type $B$ permutation tableaux}
We define a type $B_n$ {\it permutation tableau} to be a type
$(B_n,n)$ $\Le$-diagram $D$ of shape $O_w$ which contains no
all-zero column. In other words, if a column of $O_w$ has at least
one box, not all boxes in that column may be $0$ in $D$. Since
$\Phi^B_3$ is a bijection, the type $B_n$ permutation tableaux are
in bijection with the set of type $B_n$ decorated permutations such
that all fixed points in positions $1$ through $n$ are
counterclockwise.  Thus the type $B_n$ permutation tableaux are in
bijection with the set of $B_n$ permutations; in particular, there
are $2^n n!$ of them. Later, Proposition \ref{perm-tab} will give a
more refined count of the type $B_n$ permutation tableaux.

Note that if a type $B_n$ $\Le$-diagram contains an all-zero column,
then the diagonal box in that column contains a $0$, which implies
that all boxes to its left are $0$.  Deleting this ``hook" we obtain
a type $B_{n-1}$ $\Le$-diagram, and if we repeat this procedure, we
will eventually obtain a type $B$ permutation tableau.

Permutation tableaux in type $A$ were studied in \cite{SW}.  They
are simpler than $\Le$-diagrams but can be applied to the study of
permutations.  Rather surprisingly, type $A$ permutation tableaux
are related to the asymmetric exclusion process \cite{CW}.


\subsubsection{Partial order on $\I^B$}
Rietsch \cite{Rietsch2} has given a concrete description of the
order relation on cells: $P_{x,w;>0}^J \subset
\overline{P_{x',w';>0}^J}$ precisely if there exists $z \in W_J$
such that $x' \leq xz  \leq wz \leq w'$.  This poset has nice
combinatorial properties: it is thin and EL-shellable, and hence is
the poset of cells of a regular CW complex \cite{Williams2}.

Postnikov \cite{Postnikov} described this poset in the case of the
type $A$ Grassmannian in terms of decorated permutations.  One draws
decorated permutations on a circle; the cover relation involves
uncrossing two chords emanating from $i$ and $j$ that form a
``simple crossing''.  Similarly, one can describe the partial order
on $\I^B$ in the case of the type $B$ Grassmannian in terms of type
$B$ decorated permutations. The cover relation is exactly the same,
except that each time we uncross the pair of chords $i$ and $j$, we
must also uncross the pair of chords $-i$ and $-j$ (which will
necessarily also form a simple crossing).

\section{Enumeration of cells}\label{s:enumerate}
The cells in the totally non-negative Grassmannian for type $A$ were
enumerated in \cite{Williams1,Postnikov}.  In this section we will
give some formulae and recurrences for the number of totally
non-negative cells in Grassmannians of types $(B_n,1)$, $(D_n,1)$,
$(B_n,n)$, and $(D_n,n)$.  We will often count the cells which lie
inside the open Schubert cell, or in other words, we will count
$\Le$-diagrams of maximal shape.

\subsection{Enumeration of type $(B_n,1)$ and $(D_n,1)$ \protect\Le-diagrams}

Let $\hat{b}_n$  be the number of type $(B_n,1)$
 $\Le$-diagrams
of maximal shape, and let
$\hat{b}_n(q)$ be the $q$-generating function of
$(B_n,1)$ $\Le$-diagrams of maximal shape, where we weight
$\Le$-diagrams according to the number of $+$'s.  Similarly define
$\hat{d}_n$ and $\hat{d}_n(q)$.
Below, $[i]$ denotes the $q$-analog of $i$.

\begin{proposition}
$\hat{b}_n(q)$ is the coefficient of $x^n$ in
$$\frac{1-(q+q^2)x+(1-q^2)x^2}{1-[2]^2 x+[2]x^2}.$$
In particular, the
numbers $\hat{b}_n$ are equal to sequence A006012 in the Sloane
Encyclopedia of Integer Sequences \cite{Sloane}, and have generating
function $\frac{1-2x}{1-4x+2x^2}$.
\end{proposition}

\begin{proof}
It is easy to check that
$\hat{b}_0(q) = 1$, $\hat{b}_1(q) = [2]$, and
$\hat{b}_2(q)=1+2q+2q^2+q^3$.
We will
prove that for $n \geq 3$,
$\hat{b}_n(q)=(1+q)^2\hat{b}_{n-1}(q)-(1+q)q^2
\hat{b}_{n-2}(q)$,
using
the description of Theorem \ref{thm:LeBB}.
A maximal-shape $(B_n,1)$ $\Le$-diagram $D'$ can be obtained from a
maximal-shape $(B_{n-1},1)$ $\Le$-diagram $D$ by adding two boxes,
one to the left and one to the right of $D$.  Each of these two
boxes can contain a $0$ or a $+$, except that we may not put a $0$
into the new box to the right of $D$ if $D$ has a $+$ in its
leftmost and rightmost boxes.  This gives the stated recursion.
\end{proof}

Now let us consider type $(D_n,1$) $\Le$-diagrams.
For $n \geq 4$, we have the same recurrence as above:
$\hat{d}_n(q)=(1+q)^2\hat{d}_{n-1}(q)-(1+q)q^2
\hat{d}_{n-2}(q)$.  We have initial conditions
$\hat{d}_0(q)=1$,
$\hat{d}_1(q)=[2]$,
$\hat{d}_2(q)=[2]^2$,
$\hat{d}_3(q)=[2]^4-q^2[2]$.
This implies the following result.

\begin{proposition}
$\hat{d}_n(q)$ is the coefficient of $x^n$ in
$$\frac{1-(q+q^2)x+(1-2q^2-q^3)x^2+(1+2q-q^3)x^3}{1-[2]^2x+[2]x^2}.$$
In particular, the numbers $\hat{d}_n$ (for $n \geq 3$) are
given by sequence A007070 in the Sloane Encyclopedia of Integer
Sequences \cite{Sloane}.
\end{proposition}

\subsection{Enumeration of $(B_n,n)$ \protect\Le-diagrams and
permutation tableaux}

\subsubsection{The total number of $(B_n,n)$ $\Le$-diagrams }

Let $B(n)$ be the number of type $(B_n,n)$ $\Le$-diagrams; equivalently,
it is the number of type $B$ decorated permutations on
$\{\pm 1, \dots , \pm n \}$.  Let $b(n)$ be the number of elements in
the Coxeter group of type $B_n$.  That is, $b(n)=2^n n!$.

\begin{proposition}
The sequence of numbers $B(n)$ is sequence A010844 from the Sloane
Encyclopedia of Integer Sequences \cite{Sloane}. The numbers are
given by the recurrence $B(0)=1$ and $B(n+1) = 2(n+1)B(n)+1.$
\end{proposition}

\begin{proof}
A type $B$ decorated permutation is chosen via the following
procedure: first choose a number $k$ (where $0 \leq k \leq n$),
which will be the number of clockwise fixed points; then choose
their location in ${n \choose k}$ ways; finally, choose the
structure of the remainder of the permutation by choosing a normal
type $B$ permutation of size $n-k$ in $b(n-k)$ ways.  Therefore
$$B(n) = b(n)+{n \choose 1} b(n-1) + {n \choose 2} b(n-2) +\dots
                 + {n \choose n} b(0).$$

Since ${n+1 \choose k} b_{n+1-k} = 2(n+1) {n \choose k} b_{n-k}$, we have
 $B(n+1) =2(n+1)B(n)+1$.
\end{proof}

\subsubsection{The total number of $B_n$ permutation
tableaux} We say that a $0$ in a $\Le$-diagram is {\it restricted}
if it is on the diagonal or if there is a $+$ below it in the same
column.  We say that a row is restricted if it contains a restricted
$0$.  Let $t_{n,k,j}$ be the number of type $B_n$ permutation
tableaux with $k$ unrestricted rows and exactly $j$ $+$'s on the
diagonal. Let $T_n(x,y) = \sum_{k,j} t_{n,k+1,j} x^k y^j$.

\begin{proposition}\label{perm-tab}
$T_n(x,y)=(y+1)^n (x+1)(x+2)\dots (x+n-1)$.
\end{proposition}

The strategy of this proof comes from an idea in \cite{CH}.

\begin{proof}
We will show that $T_n(x,y) = (y+1)(x+1)T_{n-1}(x+1,y)$. To prove
this, let us consider the process of building a type $B_{n}$
permutation tableau $D'$ from a type $B_{n-1}$ permutation tableau
$D$ of shape $O_w$.  Let $k$ be the number of unrestricted rows of
$D$. There are two possibilities: either we can add a new (empty)
row of length $n$ to $D$ (adding a north step to the border of
$O_w$), or we can add a new column $c$ of length $n$ to $D$ (adding
a step west to the border of the Young diagram).  The first
possibility
 is easy
to analyze:  $D'$ contains the same number of $+$'s on the diagonal
as $D$ and has one additional unrestricted row.

For the second possibility there are two cases:
either we will fill the bottom (diagonal) box of $c$ with a $0$
or we will fill it with a $+$.  In each case
all boxes in a restricted row must be filled with $0$'s, and the other
boxes may be filled with $0$ or $+$.  We
need to compute
how many ways there are to fill the boxes of $c$ such that the
resulting tableau $D'$ has $i$ unrestricted rows, where $i \leq k+1$.

In the first case, note that every $0$
above the bottom-most $+$ in $c$ is restricted.  Also the $0$
is the bottom (diagonal) box of $c$ is restricted.  Summing over
the position $a$ of the bottom-most $+$, the number of ways to fill
the boxes of $c$ such that $D'$ has $i$ unrestricted rows
is $\sum_{a=1}^k {a-1 \choose k-i} = {k \choose i-1}$.

In the second case, since the bottom (diagonal) box of $c$ is a $+$,
all $0$'s that we place in column $c$ are restricted.  Therefore
there are ${k \choose i-1}$ ways to fill the boxes of $c$ such
that the $D'$ has $i$ unrestricted rows.

Our arguments show that
\begin{equation*}
t_{n,i,j} = t_{n-1,i-1,j} + \sum_{k\geq i} {k \choose i-1} t_{n-1,k,j} +
   \sum_{k\geq i-1} {k \choose i-1} t_{n-1,k,j-1},
\end{equation*}
which implies that
$t_{n,i,j} = \sum_{k \geq i-1} {k \choose i-1} (t_{n-1,k,j}+t_{n-1,k,j-1})$.
It follows that
$T_n(x,y) = (y+1)(x+1)T_{n-1}(x+1,y)$.
\end{proof}

\subsubsection{Recurrences for type $(B_n,n)$ cells of maximal shape}

Let $b(n)$ be the number of type $(B_n,n)$ $\Le$-diagrams of maximal
(staircase) shape. Let $[i] = 1 + q + \cdots + q^{i-1}$ denote the
$q$-analogue of $i$ and let $[i]^{(j)}$ denote the $j$-th derivative
(with respect to $q$) of $[i]$.  We have the following recurrence
for $b(n)$.

\begin{proposition}
We have $b(0)=1$  and
\begin{equation*}
b(n) = [n+1]b(n-1) + q^2 \sum_{i=1}^{n-2} \frac{[n-1]^{(i)}}{i!}b(n-i-1).
\end{equation*}
\end{proposition}

\begin{proof}
This result is proved by considering the various possibilities for
the top row of the $\Le$-diagram.  Whenever there is a $0$ in the
top row which is to the right of some $+$ (let us call such a $0$
{\it restricted}), then every box below that $0$ must also be a $0$.
In a type B $\Le$-diagram, whenever there is a $0$ on the diagonal,
all boxes to its left must also be $0$'s. Therefore whenever we have
a restricted $0$ in the top row, say in column $i$, then the $i$-th
column and the $n+1-i$-th row of the $\Le$-diagram contain $0$'s.
If we delete this column and this row, the resulting diagram is a
$\Le$-diagram of (inverted) staircase shape of type $B_{n-1}$.

If we $q$-count the possible configurations for the top row of a
type $B_n$ $\Le$-diagram which have no restricted $0$'s, we will get
$[n+1]$. Deleting the top row of such a diagram leaves us with a
type $B_{n-1}$ $\Le$-diagram.

If we $q$-count the possible configurations for the top row of a
type $B_n$ $\Le$-diagram which have precisely $i$ restricted $0$'s,
where $i \geq 1$, we get $q^2\frac{[n-1]^{(i)}}{i!}$.  Deleting the
top row of the $\Le$-diagram as well as the $i$ columns and rows
corresponding to the $i$ restricted $0$'s leaves us with a type
$B_{n-i-1}$ $\Le$-diagram.
\end{proof}

\subsubsection{Preference Functions} \label{sec:pref}

Let $\B_n$ denote the set
of $(B_n,n)$ $\Le$-diagrams of maximal (staircase) shape such that
the bottom square contains a $+$.  The set of $(B_n,n)$
$\Le$-diagrams of maximal shape has twice the cardinality of $\B_n$,
since the bottom square of a $\Le$-diagram imposes no restrictions
on any other square.

\begin{definition}
A {\it preference function} of $n$ is a word of length $n$ where all
the numbers $1$ through $k$ occur at least once for some $k \leq n$.
\end{definition}

In other words, a preference function of $n$ lists the possible ways
that $n$ candidates may rank in a tournament, allowing ties.

\begin{theorem}\label{thm:prefB}
The set $\B_n$ is in bijection with the set of preference functions
of length $n$. Therefore the number of maximal type $(B_n,n)$
$\Le$-diagrams has twice the cardinality of the preference functions
of length $n$. This is sequence A000629 in the Sloane Encyclopedia
of Integer Sequences \cite{Sloane}.
\end{theorem}

Theorem \ref{thm:prefB} follows from Lemmata \ref{lem:BJ} and
\ref{lem:alpha} below.
First recall that $\Phi^B_3$ gives a bijection between type $B_n$
$\Le$-diagrams of maximal shape and type $B_n$ decorated permutations
such that the nonexcedances are in positions $1, \dots , n$. So
in particular, any fixed points that occur are clockwise.
(Since they are all oriented the same way, we will ignore this
orientation from now on.)
Let $J_n$ denote the set of type $B_n$ decorated permutations
such that the nonexcedances are in positions $1, \dots, n$,
and such that there is never a fixed point in position $n$.
Restricting
$\Phi^B_3$ to $\B_n$ gives the following result.
\begin{lemma}\label{lem:BJ}
$\Phi^B_3$ is a bijection from $\B_n$ to $J_n$.
\end{lemma}

We now define a bijection $\alpha$ from $J_n$ to preference functions
of length $n$.
Let $\pi \in J_n$.  The preference function
$p=(p_1,\dots, p_n):=\alpha(\pi)$ is
defined as follows.
Let $I^+$ be the
set of indices $i$ of $\pi$ where $\pi(i)>0$.
The entries of $\pi$ in positions $I^+$ will tell us about the
``repeated" entries in the preference function.
Let $K$ be the complement of the set
$I^+ +1 :=\{i+1 \ \vert \ i\in I^+\}$ in $\{1, \dots , n\}$.
Let $S_\pi$ be the sequence that we obtain if we take the sequence
of negative entries of $\pi$, forget the signs, and then
use the relative order of the entries to extract a permutation
on $\{1,\dots,m\}$ for $m \leq n$.
We now put the entries of $S_\pi$ (in order) into
the entries $p_k$ where $k\in K$.  Finally, looking at each
$i+1\in I^+ + 1$ in increasing order, we define $p_{i+1}:=p_{\pi(i)}$.

\begin{example}
Suppose $n = 9$ and
$\pi = (-6,  -8, -3 ,-1 , -9 , 5 , -7, 4, -2)$.
Then $I^+=\{6,8\}$, $K=\{1,2,3,4,5,6,8\}$,
and $S_\pi=(4,6,3,1,7,5,2)$, and the preference function
is
$\alpha(\pi)= (4,6,3,1,7,5,7,2,1)$.
\end{example}

\begin{lemma}\label{lem:alpha}
The map $\alpha$ is a bijection from $J_n$ to the set of preference functions
of length $n$.
\end{lemma}

\begin{proof}
Since no permutation in $J_n$ has a fixed point in position $n$,
$I^+ + 1$ is a subset of $\{1,\dots , n\}$ as it should be.  Also,
the definition $p_{i+1}:=p_{\pi(i)}$ makes sense because $\pi(i)\leq
i$ by the condition on nonexcedances of $\pi$.  Therefore $\alpha$
is well-defined.

To show $\alpha$ is a bijection we will define its inverse.
 Let $K$ be the set of indices corresponding
to the first occurrence of each positive integer in
the preference function $p$.   $K$ includes $1$,
so we can reconstruct the set $I^+$ as $K^c - 1 = \{k-1 \ \vert \
k \in K^c \}$, where $K^c$ is the complement of $K$ in
$\{1,\dots , n\}$.
Now for each entry $a$ in $p$ (say in position $i+1$)
which is not the first occurrence of $a$,
we look at the closest occurrence of $a$ to the left of
position $i+1$.  Say it occurs in position $i' \leq i$.  Then
we set $\pi(i)=i'$; note that the nonexcedance condition is
satisfied.  Let $T$ be the set of all such $i'$ and
let $T^c$ be the complement of $T$.
We now complete our reconstruction of $\pi$
by placing the elements of $T^c$ in the unfilled positions
of $\pi$ in the same relative order as the first occurrences
of entries of $p$, and then negating their signs.
\end{proof}

\subsection{Enumeration of $(D_n,n)$ \protect\Le-diagrams}

In this section we show that the set $\D_n$ of maximal type
$(D_n,n)$ $\Le$-diagrams is in bijection with a distinguished subset
of preference functions, the {\it atomic} preference functions.  A
preference function is {\it atomic} if no strict leading subword
consists of the only occurrences in the word of the letters $1$
through $j<k$.

\begin{theorem}\label{th:mainD}
Atomic preference functions of length $n$ are in bijection with
maximal type $(D_n,n)$ $\Le$-diagrams.
Therefore the cardinality of the set of maximal
type $(D_n,n)$ $\Le$-diagrams is given by sequence A095989
from the Sloane Encyclopedia of Integer Sequences \cite{Sloane}.
\end{theorem}

Let $\mathcal A_n$ be the set of atomic preference functions of
length $n$.
For $D \in\mathcal D_n$ we let $0_R(D)$
denote the set of indices $i$ such that row $i$ of $D$ is
completely filled with $0$'s.

We will prove Theorem \ref{th:mainD}
by describing two maps between these sets and showing that they are inverse
to each other.

First we describe $\Phi: \mathcal D_n \to \mathcal A_n$, and then
$\Psi: \mathcal A_n \to \mathcal D_n$ which will be the inverse.

$\Phi$ is defined as follows.  Embed the $\Le$-diagram $D$ into a
staircase shape with $n$ rows by adding a diagonal to $D$ which is
filled entirely with $*$'s. Label the west side (north side) of the
staircase with the numbers from $1_W$ to $n_W$ ($1_N$ to $n_N$), as
in the following diagram:

$$\tableau*[sbY]{\bl&\bl 5_N&\bl 4_N&\bl 3_N&\bl 2_N&\bl 1_N\\\bl 1_W&&&&&*%
\\ \bl 2_W&&&&$*$\\ \bl 3_W&&&* \\ \bl 4_W&&*\\ \bl 5_W&*\\}$$

Turn each $+$ and each $*$ into an $\textelbow$  and each $0$ into a
crossing $\textcross$. This will turn our diagram into a wiring
diagram which gives a permutation $\pi = \pi(D)$ (where paths $i
\mapsto \pi(i)$ travel from the west to the north border). We now
add signs to $\pi$, making the $i$th entry positive if  $i \in
0_R(D)$; and otherwise negative. Clearly this signed permutation is
an element of the set $J_n$ defined in  Section \ref{sec:pref}; in
fact, the map we have described is essentially the map $\Phi^B_3$.
We now define $\Phi(D):=\alpha(\pi(D))$, where $\alpha$ is the map
used in Section \ref{sec:pref}.

\begin{proposition}
$\Phi(D)$ is an atomic preference function for $D\in \D_n$.
\end{proposition}

\begin{proof}
Suppose that $\Phi(D)$
is not atomic.  This means that there is a proper leading subword
(say of length $j$) of $\Phi(D)$ which consists of all occurrences of the
numbers $1$ through $r$
for some positive $r$.  Recalling the definition of $\alpha$,
this means that any negative entry of $\pi:=\pi(D)$ after position $j$
has greater absolute
value than any negative entry of $\pi(D)$ in the first $j$ positions.
Furthermore, if for any $i \in I^+$ we have $\pi(i) \leq j$
then $i+1 \leq j$.  This implies that if we ignore signs, the first
$j$ entries of $\pi(D)$ form a permutation of $S_j$, and $\pi(j)$
is negative.

Now note that since $|\pi(1)| \leq j$,  the first $n-j$ entries
in the first row of $D$ must be zero.  Similarly, since
 $|\pi(2)| \leq j$,  the first $n-j$ entries of the second
row of $D$ must be zero.  Continuing, since
$|\pi(j)| \leq j$, the first $n-j$ entries of the $j$th row
of $D$ must be zero, i.e. all entries in the $j$th row of
$D$ are zero.  But this means that $\pi(j)>0$, a contradiction.
\end{proof}

We remark that this proof did not use the forbidden patterns of
type $(D_n,n)$ $\Le$-diagrams in any way.  In fact, we can define
$\Phi$ for {\it any}
$\oplus$-diagram $D$, and $\Phi(D)$ will be an atomic preference function;
this is a many-to-one map.  What we need to prove next is that
when we restrict $\Phi$ to the set of type $(D_n,n)$ $\Le$-diagrams,
we get a bijection to the set of atomic preference functions.

\subsubsection{Inverse bijection}

We shall refer to the type $(D_n,n)$ $\Le$-conditions as the first,
second and four-box pattern.  We remark that a consequence of our
labeling of the columns and rows is that if the 0 involved in the
second or four-box $\Le$-pattern is in column $k$ then the lowest
$+$ involved in the $\Le$-pattern is in row $k$.

We will construct the inverse map $\Psi: \mathcal A_n \to \mathcal
D_n$ recursively.  Recall from Lemma \ref{lem:alpha} that given a
preference function $f$, we can already construct the signed
permutation $w(f):=\alpha^{-1}(f)$. Let $\w :=|w(f)|$.

We will first construct the path $P_n = n_N \to \w^{-1}(n)_W$, then
the path $P_{n-1} = (n-1)_N \to \w^{-1}(n-1)_W$, then $(n-2)_N \to
\w^{-1}(n-2)_W$, and so on.  The general idea here is that each path
$P_i$ will travel as close to the northwest border of the staircase
as possible.  This idea will be made precise in the form of an
algorithm in the following
paragraphs.

Let $D_i = \cup_{j=i}^n P_j$ denote the
set of boxes used by $P_i, P_{i+1}, \ldots, P_n$ so that $D_1$ is
completely filled in. Abusing notation, we also use $D_i$ to
denote the corresponding staircase shape partially filled with $0$'s
and $+$'s (diagonal boxes are always filled in with $*$'s).  Let
$C_b(D_i), R_a(D_i)$ denote the $b$-th column and $a$-th row of
$D_i$.  We say a row or a column is {\it complete} if
all its boxes have been filled in.

Let $i^* = \w^{-1}(i)$ so that $P_i$ goes from $i$ to $i^*$.  A path
$P_i$ is completely determined by its set $P_i^+$ of boxes
containing $+$'s which are linearly ordered according to the order
in which they are visited. Our construction of $P_i^+$ will always
have the form
$$
P_i^+ = (c_1, c_1', c_2, c_2', \ldots, c_k, c_k', c_{k+1}',
c_{k+\ell-1}')
$$
where $c_j \in P_i^+ \cap (D_i - D_{i+1})$ and $c_j' \in P_i^+ \cap
D_{i+1}$.  In other words, the primed boxes are old, while the
unprimed ones are newly added.  In our notation it is possible for
$\ell = 0$, or in other words, $c_k$ is the last $+$ on $P_i^+$.

We now give the construction of $P_i$ by describing $c_1, c_2,
\ldots, c_k$.
It may be helpful for the reader to look at Example \ref{ex:atomic}
alongside the description of this algorithm.

Given $c_1, c_2, \ldots, c_{j-1}$, it is clear that
$c'_{j-1}$ is determined.  Suppose $c'_{j-1} = (a,b)$ is in row $a$
and column $b$.  If $j = 1$ we set $c'_{j-1} = (a,b) = (0,i)$.  If all rows
below row $a$ have been filled in then we are already done: the path
$P_i$ is determined.

 Otherwise, let $a'$ (if it exists) be the highest row (smallest number) below $a$ and above row $b$ which
contains a $+$ and let $c^* = (a',b')$ (if it exists) be the
rightmost $+$ in row $a'$.  We have a number of mutually exclusive
cases:

\begin{enumerate}
\item[(Z)] Suppose one of the following holds:
\begin{enumerate}
\item $c^*$ does not exist and $i^* \geq b$.
\item all the rows below row $a$ have been filled in $D_{i+1}$.
\item $c^*$ exists and is equal to $(i^*,b)$.
\item $c^*$ exists and $i^* > a'$ and all the rows below and including
$a'$ are complete in $D_{i+1}$.
\end{enumerate}
Then $c_{j-1}$ is already the last new $+$ in $D_i$ and the rest of
the $+$'s on $P_i$ are determined by $D_{i+1}$.

\item[(A)]
If $c^*$ exists and $i^* < a'$, or $c^*$ does not exist and $i^* <
b$ we set $c_j = (i^*,b)$, and
\begin{equation}\label{E:T1}
\text{$c_j$ will be the last $+$ of $P_i$.}
\end{equation}



\item[(B)]
Suppose $c^*$ exists, $i^* = b$ and either $w^{-1}(i)
> 0$ or $R_b(D_{i+1})$ is filled with $0$'s.  Then $c_{j-1}$ is already the last new $+$ on $P_i$, and:
\begin{align}\label{E:T2}
&\text{the path $P_i$ will visit the diagonal square $(i^*,i^*)$ and
exit at $i^*_W$.}
\end{align}

\item[(C)]
Suppose $c^*$ exists.  If $i^* > b$, $w(b) < 0$ and
$R_{b}(D_{i+1})$ is filled with 0's apart from one box then we set
$c_j = (b,b')$. Then:
\begin{align}\label{E:T3}
&\text{the path $P_i$ will visit the diagonal square $(b,b)$ and
then turn at $c_j$;} \\
\label{E:T33} &\text{also $c_j$ is the last new $+$ of $D_i$.}
\end{align}

\item[(D)]
In all other cases we set $c_j = (a',b)$.  Then:
\begin{equation}\label{E:T4}
\text{the path $P_i$ will turn at $c_j$ and head to $c'_j = c^*$.}
\end{equation}
\end{enumerate}

Note that only in Case D does one have to continue constructing
$c_{j+1}$.  The construction of $P_i$ typically involves multiple
instances of Case D, followed by one instance of another case.

We will call a $+$ placed via Case C a {\it special} $+$ and a $+$
placed via Case D a {\it normal} $+$.  Let us call a $+$ inside some
$D_i$ a {\it corner} $+$ if (a) its row is not yet complete, and (b)
it is the rightmost $+$ in its row. We claim that

\begin{prop}\label{P:well-defined}
The algorithm described above is well-defined.  More precisely,
\begin{enumerate}
\item
the construction gives paths $P_i$ which go from $i_N$ to $i^*_W$,
\item
the positions of the new $+$'s $c_j$ are empty in $D_{i+1}$,
\item
no $+$'s are encountered while going from $c_{j-1}'$ to $c_j$,
\item
the stated facts \eqref{E:T1}, \eqref{E:T2}, \eqref{E:T3},
\eqref{E:T33} and \eqref{E:T4} hold.
\end{enumerate}
Furthermore, in Cases A,B,C,D each $c_j'$ of the form $c^*$ in the
algorithm is a corner $+$ of $D_{i+1}$.
\end{prop}

 Proposition \ref{P:well-defined} will be proved simultaneously with
the following propositions.
\begin{prop} \label{P:cols}
Let $C = C_b(D_i)$ be a column of an intermediate diagram $D_i$.
\begin{enumerate}
\item If $b < i$, then $C$ is empty.
\item If $C$ contains a corner $+$, say $c$, then this corner $+$ is unique.
Every filled square of $C$ below $c$ belongs to a complete row.
Every filled square to the right of $c$ belongs to a complete
column.
\item If $b \geq i$ and $C$ does not contain a corner $+$, then $C$
is completely filled in.
\end{enumerate}
\end{prop}

\begin{prop}\label{P:rows}
Let $R = R_a(D_i)$ be a row of an intermediate diagram $D_i$.
\begin{enumerate}
\item
If $R$ is complete then $R$ is either completely filled with 0's or
all rows below $R$ are also complete.
\item
If $R$ contains a corner $+$, then the exit $a_W$ has been used
 by a path $P_j$ for some $j \geq i$.
\end{enumerate}
\end{prop}

\begin{prop}\label{P:corner}
In any intermediate diagram $D_i$,
\begin{enumerate}
\item
all boxes weakly to the north-west of a corner $+$ or a normal $+$
are filled in; and all boxes strictly northwest of any $+$ are
filled in,
\item
corner $+$'s are arranged from north-east to south-west, forming the
corners of a (English notation) Young diagram.
\end{enumerate}
\end{prop}

\begin{prop}\label{P:int}
The set of new squares $P_i \cap (D_i - D_{i+1})$ contains at most
one box in each row, except in Case B when one has a row completely
filled with 0's.
\end{prop}

\begin{prop} \label{P:Le}
There are no violations of the type $(D_n,n)$ $\Le$-condition in any
intermediate diagram $D_i$.
\end{prop}

\begin{proof}[Proof of Propositions \ref{P:well-defined}, \ref{P:cols}, \ref{P:rows}, \ref{P:corner}, \ref{P:int} and \ref{P:Le}.
]  All the claims hold when no paths have been added.  Let us assume
that $D_{i+1}$ has been constructed and that all statements hold. We
shall show that $D_i$ satisfies all these conditions.

We first note:
\begin{enumerate}
\item
Proposition \ref{P:cols}(1) is obvious.
\item
A corner $+$ is always normal.
\item
Proposition \ref{P:rows}(1) follows from Proposition \ref{P:int}
and the following
wire-counting argument.  Let $n$ be the number of rows below and
including $R$.  If $R$ is not completely filled
with $0$'s then $n-1$ wires will travel from the north through row
$R$.  Let $R'$ be any row below $R$, say of length $\ell-1$.
Then there are  $\ell$ exits below and including $R'$ of which at least
$\ell-1$ must have been used, so $\ell-1$ wires exit to the
left below $R'$ and hence must enter and occupy every square of $R'$.
\end{enumerate}
Suppose $P_i$ has been constructed up to $c_{j-1}' = (a,b)$ where
$c_{j-1}'$ might mean the ``entrance'' $i_N = (0,i)$.  In our
explanations we will assume that $c_{j-1}' \neq (0,i)$ -- i.e.\
that $j>1$.  (Note that the
special case $c_{j-1}' = (0,i)$ is easier.)  We may assume
(inductively) that $i^* > a$ and that $c_{j-1}'$ is a corner $+$ in
$D_{i+1}$.

Case 1: Suppose that $c^* = (a',b')$ exists. By Proposition
\ref{P:corner}(2),  $c^*$ is either a corner $+$ or its
row is filled in.  Furthermore, if $b' < b$ then by Proposition
\ref{P:corner}(1) all the rows below row $a$ are complete, so we are
in Case Z(b), and nothing needs to be proved (the fact that $P_i$
will exit correctly follows from counting wires).  Suppose $b' = b$
so that the row $a'$ is complete.  If $i^* \geq a'$ then we are in
either Case Z(c) or Z(d). The only new boxes we fill are with $0$'s.
It is straightforward to verify the claimed properties.

Now suppose that $i^* < a'$.  Then we are in Case A. Consider $R =
R_{i^*}(D_{i+1})$. By Proposition \ref{P:cols}(2), the box $(i^*,b)$
is either empty or $R$ is complete.  If $R$ is complete, then
because of the way we chose $a'$, it must be filled with 0's.  But
this can be shown to be impossible by considering the wire that
passed through $(i^*,b)$. (The wire $P'$ passing through $(i^*,b)$
did not exit at $i^*_W$, since the current wire $P$ needs to use
this exit.  So this wire traveled down column $b$ through $(i^*,b)$.
But there aren't any $+$'s between $(i^*,b)$ and $(a,b)$ so $P'$
turns at $(a,b)$ which means it is the same wire as $P$, a
contradiction.) By Proposition \ref{P:corner}, Proposition
\ref{P:cols}(3) and the way we chose $a'$, we see that all the boxes
to the left of $(i^*,b)$ have been filled with $0$'s in $D_{i+1}$.
It is easy to see that Propositions \ref{P:well-defined},
\ref{P:cols}, \ref{P:rows}, \ref{P:corner}, and \ref{P:int} continue
to be satisfied in $D_i$. Since the boxes in $\{(c,d) \mid a < c
\leq i^*, b< d \leq n\}$ are all filled with 0's the first and
four-box $\Le$-conditions are immediate (for the four-box condition
one also uses Proposition \ref{P:corner}). Suppose the second
$\Le$-condition is violated by the new $+$ in box $(i^*,b)$.  There
are two possibilities: (i) the $+$ in $(i^*,b)$ is the lower $+$ in
the $\Le$-pattern, or (ii) $+$ in $(i^*,b)$ is the higher $+$ in the
$\Le$-pattern.  In case (i), the 0 in the violating pattern is in
column $i^*$ say at $(x,i^*)$ with a $+$ at $(x,y)$ where $b > y
> i^*$.  Since $(i^*,b)$ is empty in $D_{i+1}$, by Proposition
\ref{P:rows} and Proposition \ref{P:cols}, column $i^*$ is filled in
so all the squares below $(x,i)$ contain $0$'s. But it follows from
the definition of the algorithm that row $i^*$ is already complete
in $D_{i+1}$, a contradiction.  In case (ii), let the 0 in the
$\Le$-pattern be in box $(i^*,j)$.  Then $C_j(D_{i+1})$ is complete
with 0's under $(i^*,j)$ and one deduces that $R_j(D_{i+1})$ is
complete and has a single $+$, which must then be in a column to the
left of column $b$. Consider the wire $P'$ which passed vertically
through the 0 in box $(j,b)$. This wire cannot turn somewhere
between $(j,b)$ and $(a,b)$ for then the second $\Le$-condition is
already violated in $D_{i+1}$. But this means $P' = P_i$, a
contradiction.  This completes the verification of the properties in
Case A (when $c^*$ exists).

Now suppose $i^* \geq a'$.  We have already treated the case $b'
< b$ so we assume $b' \geq b$.  If $R_{a'}(D_{i+1})$ is complete,
then by Proposition \ref{P:rows}(1) all rows below are complete so
we are in Case Z(d).  The argument is again straightforward.

Otherwise, if $R_{a'}(D_{i+1})$ is not complete, then
$c^*$ is a corner $+$, $b' \geq b$ and $(a',b)$ is
empty. The squares between $c_{j-1}'$ and $(a',b)$ are either
unfilled or contain $0$'s (this comes from Proposition \ref{P:cols}(2)
and the way we chose $a'$).
Similarly, the squares
below $(a',b)$ are either unfilled or belong to complete rows.
Note that these
squares do not contain $+$'s for otherwise either $(a',b)$ would be
filled (if the closest such $+$ is normal) or $c^*$ could not be a
$+$ (if the closest such $+$ is special -- this follows from the
description of Case C below).
It is also clear from the definitions
and Proposition \ref{P:cols}(2,3) that
\begin{equation}\label{E:zeros}
\text{there are 0's between $(a',b)$ and $c^*$}.
\end{equation}

Now suppose $i^* \geq b$.  Suppose first that $R_{b}(D_{i+1})$ is
complete.  If $R_b(D_{i+1})$ is filled with $0$'s then since there
are no $+$'s under $(a',b)$ we deduce that all the exits $(b+1)_W,
(b+2)_W, \ldots, n_W$ have been used.  Thus automatically we have
$i^* = b$ and we are in Case B.  Since we are only adding 0's the
claimed properties are easy to verify, except perhaps the
$\Le$-condition.  But row $R_b$ is also filled with $0$'s so there
are no possibilities of any $\Le$-patterns.


We claim with our assumptions that $R_b(D_{i+1})$ cannot be complete
but not filled with $0$'s.  Suppose this is the case.  Then by
Proposition \ref{P:int}, $n-b+1$ wires have already been drawn
passing through row $b$. Consider the wire $P_j$ which passes
through $(b,b')$. This wire cannot also pass through $c^*$ so by
assumptions there is $x \in (a',b)$ so that $(x,b')$ contains a $+$.
This $+$ must be special, so $R_x(D_{i+1})$ is complete.  But then
there is a wire $P_k$ passing through vertically through box
$(x,b)$, which contradicts our assumptions.

Thus we suppose that $R_b(D_{i+1})$ is not filled in (but still
$i^* \geq b$).  If $i^* = b$ and $w^{-1}(i) > 0$ we are again
in Case B and it is easy to verify all the claimed properties.

Now consider Case C, so $R_b(D_{i+1})$ is filled with 0's apart from
one square. Since exit $b_W$ has not been used, $n-b-1$ of the exits
$(b+1)_W, \ldots, n_W$ have been used in $D_{i+1}$. By a counting
argument,  all rows below $R_{b}$ are filled in, so the
$+$ in $(b,b')$ is the last new $+$ on path $P_i$. This proves
Proposition \ref{P:well-defined}. The other properties are
straightforward to establish except Proposition \ref{P:Le}.  It
follows from Proposition \ref{P:cols}(2,3) and Proposition
\ref{P:rows}(1) that the region $\{(c,d) \mid b \geq c
> a , n \geq d > b\}$ is filled with $0$'s in $D_{i+1}$.  Using the
$0$'s in \eqref{E:zeros} we see that the $+$ in $(b,b')$ is not
involved in any $\Le$-conditions.  Finally we consider the new $0$'s
placed in column $b$.  Only $(a',b)$ has a $+$ to the left,  so
Proposition \ref{P:Le} follows.

In all other situations we are in Case D.  The new $+$ in $(a',b)$
becomes the new corner $+$ in column $b$, unless the row $a'$
becomes complete.  If row $a'$ becomes complete, then column $b$ is
also complete by Proposition \ref{P:rows}(1).  Again the stated
properties are immediate except Proposition \ref{P:Le}.  This last
property follows from the definition of $c^*$ (minimality of row)
and arguments similar to those in Case A.

Case 2: $c^*$ does not exist.  If $i^* < b$, then the argument is
exactly the same as in Case A when $c^*$ does exist.  So we may
suppose $i \geq b$ and we are in Case Z(a).  Consider the columns
$C_k(D_{i+1})$ for $k > b$.  By Proposition \ref{P:cols} they are
either completely filled, or contain a corner $+$.  Any corner $+$'s
in these rows must be below
or on row $b$.  But a counting argument
shows that there is not enough space to fit corner $+$'s, and thus all
rows below and including row $b$ are complete.
The argument is now
the same as the other Case Z arguments.
\end{proof}

We have shown that $\Psi$ maps atomic preference functions to
$\Le$-diagrams.  Recall that $0(D)$ denotes the set of rows of $D$
which are completely filled with 0's.

\begin{lemma}\label{lem:0}
Let $f$ be an atomic preference function with corresponding signed
permutation $w = w(f)$.  Then $0(\Psi(f)) = \{j \mid w(j) > 0\}$.
\end{lemma}
\begin{proof}
Let $D = \Psi(f)$.  Suppose $i^*$ is such that $w^{-1}(i) > 0$;
then setting $j=i^*$, we have $w(j)>0$.   Then
in particular $i^* \geq i$ since $w$ does not have an excedance at
$i^*$. By construction, up till $D_{i+1}$ no $+$'s have been placed in
row $i^*$.  If $i^* = i$ we are done. Otherwise $i^* > i$, and if
column $C_{i^*}(D_{i+1})$ has a corner $+$ we are done since it must
be encountered by the path $P_i$. Suppose otherwise, so $C_{i^*}$ is
complete (by Proposition \ref{P:cols}) and must contain a $+$ in say
$(x,y)$.  But then $R_x(D_{i+1})$ is complete, so by Proposition
\ref{P:rows}(1) so is $R_{i^*}(D_{i+1})$.

Conversely, suppose $i^*$ is such that $w^{-1}(i) < 0$; in other
words, setting
$j=i^*$, we have $w(j)<0$.  If $i^* <
i$ then we are done since the construction of $P_i$ will place a $+$
in row $R_{i^*}$ before it is complete.  So suppose $i^* \geq i$ and
that $R_{i^*}(D_{i+1})$ is completely filled with $0$'s.  If $i^* =
i$ then $f$ is not atomic so we suppose $i ^* > i$.  Let us pick $j
\in (i,i^*]$ such that $R_{i^*}(D_j)$ is complete but
$R_{i^*}(D_{j+1})$ is not. Again with the same argument as above, we
conclude that $C_{i^*}(D_{j+1})$ has a corner $+$, say $c^*$.  Thus
the path $P_j$ travels through $c^*$, at which time it will enter
Case C and create a $+$ in row $i^*$.
\end{proof}

\begin{theorem}
The map $\Psi: \mathcal A_n \to \mathcal D_n$ is a bijection.
\end{theorem}
\begin{proof}
It follows from the construction (Proposition
\ref{P:well-defined}(1)) that $\Phi \circ \Psi$ is the identity, so
it suffices to show that $\Phi$ is injective.  For an atomic
preference function $f$, we will show that there are no choices in
the construction of $\Psi(f)$ if we require that $\Psi(f)$ is a
$\Le$-diagram $D$ satisfying $\Phi(\Psi(f)) = f$ and satisfying the
condition of Lemma \ref{lem:0} that all $0$ rows
correspond exactly to the $j$ such
that $w(j)>0$.

We may suppose by induction that there are no choices for the
construction of $D_{i+1} = \cup_{j = i+1}^n P_j$.  Now suppose we
have constructed the $+$'s of $P_i^+$ up to $c'_{j-1}$ as in the
stated algorithm.  We will show that the stated algorithm is the
only possible way to extend $P_i$, using the notation and explicit
descriptions given in the proof of the Propositions.

In Cases Z and A we have no choice if we require $P_i$ exits at
$i^*_W$.  So we may assume we are in Cases B, C, or D and that $c^*$
exists.  In particular $(a',b)$ is empty. For $P_i$ to exit at
$i^*_W$ we must fill any unfilled boxes between $c'_{j-1}$ and
$(a',b)$ with $0$'s.  If $i^* = b$ and $w^{-1}(i) < 0$ then the only
way for row $i$ to be completely filled with $0$'s is for $P_i$ to
go to the diagonal and then go straight to $i^*_W$ without turning.
Alternatively, if $R_b(D_{i+1})$ is complete and $i^* \geq b$ then
by the proof of the algorithm we must have $i^* = b$ and
$R_b(D_{i+1})$ filled with 0's.  This shows that Case B is forced.

Otherwise we are in Cases C or D. The first choice is thus $(a',b)$.
Suppose we place a $0$ in $(a',b)$.  Then using the first
$\Le$-condition we see that all boxes below $(a',b)$ must also be
filled with $0$. Since the boxes between $c^*$ and $(a',b)$ are
filled with $0$'s we see using the first and second $\Le$-conditions
that the only place for a $+$ in row $b$ is in box $(b,b')$.  But we
must not have row $i^*$ being completely filled with $0$'s,
otherwise we would be in Case B.  Thus we must turn at $(b,b')$.  We
claim that this is exactly Case C.  It is clear that $i^* > b$.  We
need to show that row $i^*$ in $D_{i+1}$ is filled (necessarily with
$0$'s) except the box $(b,b')$ which is empty.  The columns
$C_k(D_{i+1})$ for $b' > k > b$ do not contain corner $+$'s so by
Proposition \ref{P:cols} they are complete.  The columns
$C_k(D_{i+1})$ for $k < b'$ cannot contain $+$'s in the rows between
$a'$ and $b$, so they are either complete or contain a corner $+$
below row $b$.  Thus row $b$ contains $0$'s in all boxes except
$(b,b')$ in $D_{i+1}$.  If row $b$ is complete in $D_{i+1}$, then a
wire-counting argument shows that exit $i^*_W$ has been used.  But
the wire passing through $i^*_W$ cannot have gone straight from
$c^*$ to $(b,b')$ for otherwise $(a',b)$ would not be empty in
$D_{i+1}$. Thus there must be a complete row between rows $a'$ and
$b$, which contradicts either Proposition \ref{P:int} or the fact
that the current wire will travel down column $b$ to the diagonal
from $(a',b)$.

Thus when there is a choice, a $0$ is placed in $(a',b)$ only in
Case C.  However, if the diagram satisfies the conditions of Case C
and we place a $+$ in $(a',b)$ instead then the wire $P_i$ will exit
in row $b$, contradicting the estimate $i^* > b$.  Thus Case C is
forced by our assumptions.  In all other cases, we will perform Case
D.

\end{proof}

\subsection{Example}\label{ex:atomic}
Suppose $n = 9$ and $f = (4,6,3,1,7,5,7,2,1)$.
Then $\w = w(f) = -6 \ -8 \ -3 \ -1 \ -9 \ 5 \ -7
\ 4 \ -2$. The construction of $\Psi(f)$ proceeds as follows:

First  $i=9$ and $i^*=5$.
When $j=1$, we have $(a,b)=(0,9)$ and $c^*$ does not exist.  We are
in Case A so $c_1=(5,9)$ and $D_9$ is as shown below.
$$\tableau*[sbY]{\bl&\bl 9_N&\bl 8_N&\bl 7_N&\bl 6_N&\bl 5_N&\bl 4_N&\bl 3_N&\bl 2_N&\bl 1_N\\
\bl 1_W&\tf 0 &&&&&&&&*\\
\bl 2_W&\tf 0 &&&&&&&*\\
\bl 3_W&\tf 0&&&&&&*\\
\bl 4_W&\tf 0&&&&&*\\
\bl 5_W&\tf +&&&&*\\
\bl 6_W&&&&*\\
\bl 7_W&&&*\\
\bl 8_W&&*\\
\bl 9_W&*\\
}$$

Now $i=8$ and $i^*=2$.  When $j=1$, we have $(a,b)=(0,8)$,
$a'=5$, and $c^*=(5,9)$.  We are in Case A
so $c_1 = (2,8)$ and $D_8$ is below.
$$\tableau*[sbY]{\bl&\bl 9_N&\bl 8_N&\bl 7_N&\bl 6_N&\bl 5_N&\bl 4_N&\bl 3_N&\bl 2_N&\bl 1_N\\
\bl 1_W&0&\tf 0&&&&&&&*\\
\bl 2_W&0&\tf +&&&&&&*\\
\bl 3_W&0&&&&&&*\\
\bl 4_W&0&&&&&*\\
\bl 5_W&+&&&&*\\
\bl 6_W&&&&*\\
\bl 7_W&&&*\\
\bl 8_W&&*\\
\bl 9_W&*\\
}$$

Now $i=7$ and $i^*=7$.  When $j=1$, $(a,b)=(0,7)$, $a'=2$, and
$c^*=(2,8)$.  This is Case D so $c_1=(2,7)$ and $c'_1=(2,8)$.
When $j=2$, $c'_1=(2,8)=(a,b)$, $a'=5$, and $c^*=(5,9)$.
This is Case D so $c_2=(5,8)$ and $c'_2=(5,9)$.
When $j=3$, $(a,b)=(5,9)$.  Neither $a'$ nor $c^*$ exist so
we are in Case A, $c_3=(7,9)$, and $D_7$ is below.
$$\tableau*[sbY]{\bl&\bl 9_N&\bl 8_N&\bl 7_N&\bl 6_N&\bl 5_N&\bl 4_N&\bl 3_N&\bl 2_N&\bl 1_N\\
\bl 1_W&0&0&\tf 0&&&&&&*\\
\bl 2_W&0&+&\tf +&&&&&*\\
\bl 3_W&0&\tf 0&&&&&*\\
\bl 4_W&0&\tf 0&&&&*\\
\bl 5_W&+&\tf +&&&*\\
\bl 6_W&\tf 0&&&*\\
\bl 7_W&\tf +&&*\\
\bl 8_W&&*\\
\bl 9_W&*\\
}$$

Now $i=6$ and $i^*=1$.  When $j=1$, $(a,b)=(0,6)$, $a'=2$,
and $c^*=(2,7)$.  We are in Case A so $c_1=(1,6)$ and $D_6$ is
below.
$$\tableau*[sbY]{\bl&\bl 9_N&\bl 8_N&\bl 7_N&\bl 6_N&\bl 5_N&\bl 4_N&\bl 3_N&\bl 2_N&\bl 1_N\\
\bl 1_W&0&0&0&\tf +&&&&&*\\
\bl 2_W&0&+&+&&&&&*\\
\bl 3_W&0&0&&&&&*\\
\bl 4_W&0&0&&&&*\\
\bl 5_W&+&+&&&*\\
\bl 6_W&0&&&*\\
\bl 7_W&+&&*\\
\bl 8_W&&*\\
\bl 9_W&*\\
}$$

Now $i=5$ and $i^*=6$.  When $j=1$, we have $(a,b)=(0,5)$,
$a'=1$, and $c^*=(1,6)$.  This is Case D so $c_1=(1,6)$ and
$c'_1=(1,6)$.  When $j=2$ we have $(a,b)=(1,6)$, $a'=2$,
and $c^*=(2,7)$.  This is Case B so there are no new $+$'s;
$D_5$ is below.
$$\tableau*[sbY]{\bl&\bl 9_N&\bl 8_N&\bl 7_N&\bl 6_N&\bl 5_N&\bl 4_N&\bl 3_N&\bl 2_N&\bl 1_N\\
\bl 1_W&0&0&0&+&\tf +&&&&*\\
\bl 2_W&0&+&+&\tf 0&&&&*\\
\bl 3_W&0&0&&\tf 0&&&*\\
\bl 4_W&0&0&&\tf 0&&*\\
\bl 5_W&+&+&&\tf 0&*\\
\bl 6_W&0&\tf 0&\tf 0&*\\
\bl 7_W&+&&*\\
\bl 8_W&&*\\
\bl 9_W&*\\
}$$

Now $i=4$ and $i^*=8$.  When $j=1$, $(a,b)=(0,4)$, $a'=1$,
and $c^*=(1,5)$.  This is Case D so $c_1=(1,4)$ and $c'_1=(1,5)$.
When $j=2$, $(a,b)=(1,5)$, $a'=2$, and  $c^*=(2,7)$.  This is again
Case D so $c_2=(2,5)$ and $c'_2=(2,7)$.  When $j=3$,
$(a,b)=(2,7)$, $a'=5$, and $c^*=(5,8)$.  This is still Case D so
$c_3=(5,7)$ and $c'_3=(5,8)$.  When $j=4$, $(a,b)=(5,8)$,
$a'=7$, and $c^*=(7,9)$.  This is Case B so there are no new
$+$'s; $D_4$ is below.
$$\tableau*[sbY]{\bl&\bl 9_N&\bl 8_N&\bl 7_N&\bl 6_N&\bl 5_N&\bl 4_N&\bl 3_N&\bl 2_N&\bl 1_N\\
\bl 1_W&0&0&0&+&+&\tf +&&&*\\
\bl 2_W&0&+&+&0&\tf +&&&*\\
\bl 3_W&0&0&\tf 0&0&&&*\\
\bl 4_W&0&0&\tf 0&0&&*\\
\bl 5_W&+&+&\tf +&0&*\\
\bl 6_W&0&0&0&*\\
\bl 7_W&+&\tf 0&*\\
\bl 8_W&\tf 0&*\\
\bl 9_W&*\\
}$$

Now $i=3$ and $i^*=3$.  When $j=1$, $(a,b)=(0,3)$, $a'=1$, and
$c^*=(1,4)$.  This is Case D so $c_1=(1,3)$ and $c'_1=(1,4)$.
When $j=2$, $(a,b)=(1,4)$, $a'=2$, and $c^*=(2,5)$.  This is
Case D so $c_2=(2,4)$ and $c'_2=(2,5)$.  When $j=3$,
$(a,b)=(2,5)$, $a'=5$, and $c^*=(5,7)$.  This is Case A so
$c_3 = (3,5)$; $D_3$ is below.
$$\tableau*[sbY]{\bl&\bl 9_N&\bl 8_N&\bl 7_N&\bl 6_N&\bl 5_N&\bl 4_N&\bl 3_N&\bl 2_N&\bl 1_N\\
\bl 1_W&0&0&0&+&+&+&\tf +&&*\\
\bl 2_W&0&+&+&0&+&\tf +&&*\\
\bl 3_W&0&0&0&0&\tf +&&*\\
\bl 4_W&0&0&0&0&&*\\
\bl 5_W&+&+&+&0&*\\
\bl 6_W&0&0&0&*\\
\bl 7_W&+&0&*\\
\bl 8_W&0&*\\
\bl 9_W&*\\
}$$

Finally $i=2$ and $i^*=9$.  When $j=1$, $(a,b)=(0,2)$, $a'=1$, and
$c^*=(1,3)$.  This is Case D so $c_1 =(1,2)$ and $c'=(1,3)$.  When
$j=2$, $(a,b)=(1,3)$, $a'=2$, and $c^*=(2,4)$.  This is Case D so
$c_2 = (2,3)$ and $c'_2 = (2,4)$.  When $j=3$, $(a,b)=(2,4)$,
$a'=3$, and $c^*=(3,5)$.  This is Case C so $c_3=(4,5)$;
$D_2$ is below.
$$\tableau*[sbY]{\bl&\bl 9_N&\bl 8_N&\bl 7_N&\bl 6_N&\bl 5_N&\bl 4_N&\bl 3_N&\bl 2_N&\bl 1_N\\
\bl 1_W&0&0&0&+&+&+&+&\tf +&*\\
\bl 2_W&0&+&+&0&+&+&\tf +&*\\
\bl 3_W&0&0&0&0&+&\tf 0&*\\
\bl 4_W&0&0&0&0&\tf +&*\\
\bl 5_W&+&+&+&0&*\\
\bl 6_W&0&0&0&*\\
\bl 7_W&+&0&*\\
\bl 8_W&0&*\\
\bl 9_W&*\\
}$$


\end{document}